\RequirePackage{fix-cm}
\documentclass[smallextended]{svjour3}       
\smartqed  
\usepackage{stmaryrd}
\usepackage{bbm}
\usepackage[dvipsnames]{xcolor}
\usepackage{amsmath}
\usepackage{amssymb}
\usepackage{mathtools}
\usepackage{subcaption}
\usepackage{etoolbox}
\usepackage{ulem}
\usepackage{environ}
\usepackage{ulem}
\usepackage{multicol}
\usepackage{listings}
\usepackage{graphicx}
\usepackage{multirow}
\usepackage{subcaption}
\usepackage{blkarray}
\usepackage{algorithm}
\usepackage{algpseudocode}

\usepackage{lineno,hyperref}
\modulolinenumbers[5]

\newcommand{\conv}{\operatorname{conv}}
\newcommand{\ext}{\operatorname{ext}}

\newcommand{\dist}{\operatorname{dist}}
\newcommand{\merged}{\operatorname{merged}}
\newcommand{\flag}{\operatorname{flag}}
\newcommand{\set}{\operatorname{set}}

\newcommand{\relint}{\operatorname{relint}}

\newcommand{\True}{\operatorname{True}}
\newcommand{\False}{\operatorname{False}}

\newcommand{\pre}{\operatorname{pre}}

\newcommand{\bc}{\operatorname{bc}}

\newcommand{\Proj}{\operatorname{Proj}}
\newcommand{\select}{\operatorname{select}}

\newcommand{\disjoint}{\operatorname{disjoint}}

\newcommand{\PP}{\operatorname{P}}
\newcommand{\NP}{\operatorname{NP}}
\newcommand{\CDC}{\operatorname{CDC}}
\newcommand{\MID}{\operatorname{mid}}
\newcommand{\SOS}{\operatorname{SOS}}
\newcommand{\diameter}{\operatorname{diameter}}

\usepackage{tikz,tkz-graph}
\tikzstyle{vertex}=[circle, draw, inner sep=0pt, minimum size=1.5em]
\tikzstyle{svertex}=[draw, inner sep=0pt, minimum size=1.5em]

\newcommand{\svertex}{\node[svertex]}
\usetikzlibrary{decorations}
\usetikzlibrary{arrows.meta,patterns}
\begin{document}

\title{Modeling Combinatorial Disjunctive Constraints via Junction Trees
}


\author{Bochuan Lyu \and Illya V. Hicks \and Joey Huchette}


\institute{Bochuan Lyu, Corresponding author \at
              Department of Computational and Applied Mathematics, Rice University\\
              Houston, TX, USA \\
              \email{bl46@rice.edu}           
            \and
            Illya V. Hicks \at
            Department of Computational and Applied Mathematics, Rice University \\
              Houston, TX, USA \\
              \email{ivhicks@rice.edu}
            \and
           Joey Huchette \at
            Google Research \\
            Mountain View, CA, USA\\
            \email{jhuchette@google.com}
}

\date{Received: date / Accepted: date}

\maketitle


\begin{abstract}
We introduce techniques to build small ideal mixed-integer programming (MIP) formulations of combinatorial disjunctive constraints (CDCs) via the independent branching scheme. We present a novel pairwise IB-representable class of CDCs, CDCs admitting junction trees, and provide a combinatorial procedure to build MIP formulations for those constraints. Generalized special ordered sets ($\SOS k$) can be modeled by CDCs admitting junction trees and we also obtain MIP formulations of $\SOS k$. Furthermore, we provide a novel ideal extended formulation of any combinatorial disjunctive constraints with fewer auxiliary binary variables with an application in planar obstacle avoidance.
\end{abstract}

\section{Introduction}

The study of disjunctive constraints originated in the late 1970s from Balas~\cite{balas1975disjunctive,balas1979disjunctive,balas1998disjunctive}, where he obtained extended formulations for general disjunctive programming, studied cutting plane methods, and characterized the convex hull of feasible points. Thereafter, modeling disjunctive constraints~\cite{balas1979disjunctive} has become an important topic in mixed-integer programming (MIP) with numerous applications, such as chemical engineering~\cite{biegler1997systematic,karuppiah2006global,raman1994modelling}, robotics~\cite{dai2019global,deits2014footstep,kuindersma2016optimization}, portfolio optimization~\cite{bertsimas2009algorithm,chang2000heuristics,vielma2008lifted}, and scheduling~\cite{balas1969machine,pinto1995continuous}. A disjunctive constraint has the form of 
\begin{align} \label{eq:dc_poly}
    x \in \bigcup_{i=1}^d P^i,
\end{align} 
\noindent where each $P^i$ is a polyhedron. In our work, we focus on the case where each is bounded, i.e. a polytope. Then, each $P^i$ can also be expressed as the convex combination of the finite set of its extreme points $V^i$ by the Minkowski-Weyl Theorem~\cite{minkowski1897allgemeine,weyl1934elementare}: 
\begin{align} \label{eq:poly_convex}
    P^i = \conv(V^i) := \left\{\sum_{v \in V^i} \lambda_v v: \sum_{v \in V^i} \lambda_v = 1, \lambda \geq 0 \right\}.
\end{align}

By only keeping the combinatorial structure in the disjunctive constraint, we can model the continuous variables $\lambda$ on a collection of indices $\mathcal{S}$: $\mathcal{S} := \{S^i\}_{i=1}^d$ where each extreme points has a unique index and each $S^i$ contains all indices of extreme points of $P^i$. A \textit{combinatorial disjunctive constraint} (CDC) represented by the sets $\mathcal{S}$ is 
\begin{align} \label{eq:cdc}
    \lambda \in \CDC(\mathcal{S}) := \bigcup_{S \in \mathcal{S}} Q(S),
\end{align}

\noindent where $Q(S) := \{\lambda \in \mathbb{R}^J: \sum_{v \in J} \lambda = 1,  \lambda_{J \setminus S} = 0, \lambda \geq 0\}$ and $J := \cup_{S \in \mathcal{S}} S$. The collection of index sets $\mathcal{S}$ is \textit{irredundant} if any two distinct sets in $\mathcal{S}$ are not subsets of each other. 

There are several different approaches to model~\eqref{eq:cdc}. We say a MIP formulation for $\CDC(\mathcal{S})$ is \textit{extended} if it requires auxiliary continuous variables other than $\lambda$. Otherwise, it is \textit{non-extended}. Furthermore, a MIP formulation is \textit{ideal} if each extreme point of its linear programming (LP) relaxation also satisfies the integrality conditions in the MIP formulation. An example of non-extended MIP formulation for $\CDC(\mathcal{S})$~\cite{huchette2019combinatorial,vielma2015mixed} is
\begin{subequations} \label{form:cdc_nonideal}
\begin{alignat}{2}
    & \lambda_v \leq \sum_{S \in \mathcal{S}: v \in S} z_S, & \forall v \in J \\
    &\sum_{S \in \mathcal{S}} z_S = 1, & z \{0, 1\}^{\mathcal{S}} \\
    &\sum_{v \in J} \lambda = 1, & \lambda \geq 0. \\
\end{alignat}
\end{subequations}

However, it is not necessarily ideal. Another existing formulation for combinatorial disjunctive constraints is written down explicitly by Huchette and Vielma~\cite{huchette2017nonconvex} based on the idea of Jeroslow and Lowe~\cite{jeroslow1984modelling}. A MIP formulation for $\CDC(\mathcal{S})$ can be written as

\begin{subequations} \label{form:cdc_jl}
\begin{alignat}{2}
    &\lambda_v = \sum_{S \in \mathcal{S}: v \in S} \gamma_v^{S}, & \forall v \in J \\
    &z_S = \sum_{v \in S} \gamma_v^S, & \forall S \in \mathcal{S} \\
    &\sum_{S \in \mathcal{S}} z_S = 1 & \\
    &\gamma^S \in \Delta^S, & \forall S \in \mathcal{S} \\
    &z \in \{0, 1\}^{\mathcal{S}}.
\end{alignat}
\end{subequations}

It is ideal, but it also requires $O(|\mathcal{S}| |J|)$ auxiliary continuous variables besides $O(|\mathcal{S}|)$ binary variables. As pointed out by Huchette and Vielma~\cite{huchette2019combinatorial}, another ideal MIP formulation with fewer auxiliary binary variables can be obtained~\cite[Proposition 9.3]{vielma2018embedding}:
\begin{subequations} \label{form:cdc_vielma}
\begin{alignat}{2}
    & \lambda_v = \sum_{S \in \mathcal{S}: v \in S} \gamma^S_v, & \forall v \in J \\
    & \sum_{S \in \mathcal{S}} \sum_{v \in S} \gamma_v^S = 1 & \\
    & \sum_{S \in \mathcal{S}} \sum_{v \in S} h^S \gamma_v^S = z & \\
    & \gamma^S \geq 0, & \forall S \in \mathcal{S} \\
    & z \in \{0, 1\}^r,
\end{alignat}
\end{subequations}

\noindent where $\{h^S\}_{S \in \mathcal{S}} \subseteq \{0, 1\}^r$ are distinct binary vectors so the minimum value of $r$ can be $\lceil \log_2(|\mathcal{S}|) \rceil$. 

Another generic approach to build MIP formulations for \eqref{eq:cdc} is the \textit{independent branching} (IB) scheme framework introduced by Vielma and Nemhauser~\cite{vielma2011modeling} and generalized by Huchette and Vielma~\cite{huchette2019combinatorial}. In this framework, we rewrite \eqref{eq:cdc} as $t$ intersections of $k$ alternatives each:
\begin{align} \label{eq:k_IB}
    \CDC(\mathcal{S}) = \bigcap_{j=1}^t \left( \bigcup_{i=1}^k Q(L^j_i) \right).
\end{align}

When~\eqref{eq:cdc} is \textit{pairwise IB-representable}, i.e. it can be rewritten as the form in~\eqref{eq:k_IB} with $k=2$, Vielma and Nemhauser~\cite{vielma2011modeling} provide a ideal, non-extended formulation for~\eqref{eq:cdc}. This $k=2$ IB-scheme framework can provide more computationally efficient MIP formulations of pairwise IB-representable CDC in~\eqref{eq:cdc} than either~\eqref{form:cdc_nonideal} and~\eqref{form:cdc_jl}. The next natural question is how to build ideal and non-extended MIP formulations where the numbers of binary variables and constraints can be minimized. Huchette and Vielma~\cite{huchette2019combinatorial} show that both quatities are $O(t)$ with the depth $t$ in~\eqref{eq:k_IB} and solving minimum depth $t$ is equivalent to finding a minimum biclique cover on the associated conflict graph.

However, solving the minimum biclique cover problem (MBCP) is a well-known NP-complete problem on general simple graphs~\cite{orlin1977contentment} and even on chordal bipartite graphs~\cite{muller1996edge}. MBCP is also not approximable in polynomial time to less than $O(|V|^{1 - \epsilon})$ or $O(|E|^{1/2 - \epsilon})$ factor for any $\epsilon > 0$ unless $\PP = \NP$~\cite{chalermsook2014nearly}. Despite the hardness of solving MBCP on the general graphs, the biclique cover number of some well-structured graphs are known, such as complete graphs, $2n$-vertex crown graphs~\cite{de1981boolean}, and grid graphs~\cite{guo2018biclique}. Furthermore, some polynomial-time solutions are developed by focusing on specific types of graphs, like C4-free graphs~\cite{muller1996edge}, and domino-free graphs~\cite{amilhastre1998complexity}. Although many pairwise IB-representable CDCs have logarithmic-sized MIP formulations, there are no theoretical guarantees that the number of binary variables in MIP formulations provided by pairwise IB-scheme is upper bounded by $|\mathcal{S}|$\footnote{The upper bounds of the size of minimum biclique cover of a conflict graph with $n := |\bigcup_{S \in \mathcal{S}} S|$ vertices are the vertex cover number and $n -  \lfloor \log_2(n)\rfloor + 1$~\cite{tuza1984covering}; Neither quantity is upper bounded by $|\mathcal{S}|$.}. This motivates us to focus on a specific class of pairwise IB-representable combinatorial disjunctive constraints and to design a heuristic for MBCP on graphs associated with those constraints.

We are particularly interested in a class of pairwise IB-representable disjunctive constraints: Special Ordered Set (SOS) constraints first introduced by Beale and Tomlin~\cite{beale1970special}. The two standard $\SOS$ constraints are $\SOS 1$: a constraint for a set of non-negative variables such that at most one variable can be nonzero; $\SOS 2$: a constraint for a set of non-negative variables such that at most two consecutive variables can be nonzero. Formulations of $\SOS 1$ and $\SOS 2$ have been studied by many recent works:~\cite{adams2012base,muldoon2013ideal,vielma2010mixed,vielma2011modeling}, who show that those formulations can be modeled by a logarithmic number of auxiliary binary variables. A generalization of $\SOS 1$ and $\SOS 2$ is $\SOS k$, where at most $k$ consecutive variables in the set can be nonzero~\cite{huchette2019combinatorial}.

\paragraph{Our contributions}

\begin{enumerate}
    \item We prove that a combinatorial disjunctive constraint $\CDC(\mathcal{S})$ is pairwise IB-representable if it admits a \textit{junction tree} (as defined to come in Section~\ref{sec:junction_tree}), a particular graph decomposition know as a \textit{junction tree}.
    \item We can check whether a $\CDC$ admits a junction tree in time polynomial in $|\mathcal{S}|$ and $|J|$.
    \item We propose a polynomial-time approach to find an ideal mixed-integer programming formulation of $\CDC(\mathcal{S})$ admitting a junction tree by introducing at most $|\mathcal{S}| - 1$ binary variables and $2(|\mathcal{S}| - 1)$ constraints.
    \item We provide an novel ideal, non-extended formulation of $\SOS k(N)$ with at most $\lceil \log_2(N - k + 1) \rceil + k - 2$ auxiliary binary variables for integers $N > k \geq 2$, where $N$ is the number of variables in the $\SOS k$ constraint. We show that this improves upon an existing bound~\cite{huchette2019combinatorial} for any such values of $N$ and $k$. Furthermore, if we make an additional assumption that $k > C \lceil \log_2(N) \rceil$ for some $C > \frac{1}{2}$, then the number of binary variables needed by our new formulation is at most $\frac{C+1}{3C}$ of the number required by the existing formulation in~\cite{huchette2019combinatorial}. 
    \item We present a novel ideal extended formulation based on a junction tree of any CDC. In other words, we can transfer any $\CDC(\mathcal{S})$ into one admitting a junction tree with a cost of introducing additional continuous variables no more than $\sum_{S \in \mathcal{S}} |S| - |\bigcup_{S \in \mathcal{S}} S|$. We show this new formulation uses fewer auxiliary continuous variables in the planar obstacle avoidance application than the existing formulation~\eqref{form:cdc_vielma}.
\end{enumerate}

Recently, Kis and Horv{\'a}th \cite{kis2021ideal} also study the MIP formulation of disjunctive constraints admitting a network representation that is parallel to our study, where they discover many non-trivial facet-defining inequalities for $\SOS 2$ and $\SOS k$. Consider $\SOS k$ on variables $\lambda \in \mathbb{R}^N$ such that $\sum_{i=1}^N \lambda_i = 1$, $\lambda \geq 0$, and $\lambda$ has at most $k$ consecutive non-zero components. They write down a valid but not necessarily ideal MIP formulation:
\begin{subequations} \label{eq:sosk_kis}
\begin{alignat}{2}
    &\lambda_j - \sum_{i = \max\{j-k+1, 1\}}^j z_i \leq 0,\qquad & j = 1, \hdots, N, \\
    &\sum_{j=1}^N \lambda_j = 1, \\
    & \sum_{i=1}^{N-k+1} z_i = 1, \\
    & \lambda \geq 0, z \in \{0,1\}^N.
\end{alignat}
\end{subequations}

Then, they observe a set of non-trivial facets for~\eqref{eq:sosk_kis} from a corresponding network representation~\cite{kis2021ideal}. We observe that there are $N$ binary variables in~\eqref{eq:sosk_kis}, whereas we provide an ideal and non-extended MIP formulation with $\lceil \log_2(N - k + 1) \rceil + k - 2$ binary variables when $N > k \geq 2$.





\section{Preliminaries}

We first want to introduce some basic graph notations. A \textit{simple graph} is a pair $G := (V, E)$ where $V$ is a finite set of vertices and $E \subseteq \{uv: u, v\in V, u \neq v\}$. We use $V(G)$ and $E(G)$ to represent the vertex set and edge set of the graph $G$. A \textit{subgraph} $G' := (V', E')$ of $G$ is a graph where $V' \subseteq V$ and $E' \subseteq \{uv \in E: u, v \in V'\}$. An \textit{induced subgraph} of $G$ by only keeping vertices $A$ is denoted as $G(A) = (A, E_A)$, where $E_A = \{uv \in E: u, v \in A\}$. A graph is a \textit{cycle} if the vertices and edges are $V = \{v_1, v_2, \hdots, v_n\}$ and $E = \{v_1v_2, v_2v_3, \hdots, v_{n-1}v_n, v_nv_1\}$. A graph is a \textit{path} if the vertices and edges are $V = \{v_1, v_2, \hdots, v_n\}$ and $E = \{v_1v_2, v_2v_3, \hdots, v_{n-1}v_n\}$. A graph $G := (V, E)$ is \textit{connected} if there exists a path between $u$ and $v$ for any $u, v \in V$. A graph is \textit{tree} if it is connected and does not have any subgraph that is a cycle. A \textit{spanning tree} of a graph $G$ is a subgraph of $G$ that is tree whose vertex set equals that of $G$. We refer readers to~\cite{bondy2008graph} for further general graph theory background and definitions. 


We then formally define pairwise IB-representable, infeasible sets and conflict graphs of $\CDC(\mathcal{S})$. Note that $\mathbb{R}^N_{\geq 0} := \{x \in \mathbb{R}^N: x \geq 0\}$, $\llbracket N\rrbracket := \{1,2,\hdots, N\}$, and $\llbracket N_1, N_2 \rrbracket := \{N_1, \hdots, N_2\}$. We denote that $\Delta^N := \{\lambda \in \mathbb{R}^N_{\geq 0}: \sum_{i=1}^N \lambda_i = 1\}$ when $N$ is a positive integer and $\Delta^J := \{\lambda \in \mathbb{R}^{|J|}_{\geq 0}: \sum_{j \in J} \lambda_j = 1\}$.

\begin{definition}[pairwise IB-representable]
A combinatorial disjunctive constraint $\CDC(\mathcal{S})$ is \textit{pairwise IB-representable} if it can be written as
\begin{align}
\CDC(\mathcal{S}) = \bigcap_{j=1}^t \left( Q(L^j) \bigcup Q(R^j) \right),
\end{align}
\noindent for some $L^j, R^j \subseteq J$. We denote that $\{\{L^j, R^j\}\}_{j=1}^t$ is a \textit{pairwise IB-scheme} for $\CDC(\mathcal{S})$.
\end{definition}

\begin{definition}[feasible and infeasible sets] \label{def:feasible_sets} A set $S \subseteq J$ is a \textit{feasible set} with respect to $\CDC(\mathcal{S})$ if $S \subseteq T$ for some $T \in \mathcal{S}$. It is an \textit{infeasible set} otherwise. A \textit{minimal infeasible set} is an infeasible set $S \subseteq J$ such that any proper subset of $S$ is a feasible set.
\end{definition}

\begin{definition}[conflict graphs]
A \textit{conflict graph} for a $\CDC(\mathcal{S})$ is denoted as $G^c_{\mathcal{S}} := (J, \bar{E})$ with $\bar{E} := \{\{u, v\} \in J \times J: u \neq v, \{u, v\} \text{ is an infeasible set}\}$.
\end{definition}


\begin{proposition}[Theorem 1~\cite{huchette2019combinatorial}\footnote{We only consider the case when $k=2$ and we use minimal infeasible set directly without defining a hypergraph as in \cite{huchette2019combinatorial}.}] \label{prop:pairwise_IB_at_most_two}
A pairwise IB-scheme exists for $\CDC(\mathcal{S})$ if and only if each minimal infeasible set has cardinality at most 2.
\end{proposition}




\begin{proposition}[Theorem 4~\cite{vielma2011modeling} and Proposition 2~\cite{huchette2019combinatorial}] \label{prop:IB_MIP}
Given a pairwise IB-scheme $\{\{L^j, R^j\}\}^t_{j=1}$ for $\CDC(\mathcal{S})$, the following is a ideal formulation for $\CDC(\mathcal{S})$: given $\lambda \in \Delta^J$,
\begin{subequations}
\begin{alignat*}{2}
    & \sum_{v \not\in L^j} \lambda_v \leq z_j, \qquad \sum_{v \not\in R^j} \lambda_v \leq 1 - z_j, \qquad z_j \in \{0, 1\}, \qquad & \forall j \in \llbracket t \rrbracket.
\end{alignat*}
\end{subequations}
\end{proposition}

From Proposition~\ref{prop:IB_MIP}, the size of the MIP formulation of $\CDC(\mathcal{S})$ obtained by a pairwise IB-scheme $\{\{L^j, R^j\}\}^t_{j=1}$ is directly determined by $t$, i.e. the numbers of binary variables and constraints are both $O(t)$. Hence, the next natural question to building a small MIP formulation is how to obtain a pairwise IB scheme with the smallest depth $t$. Huchette and Vielma~\cite{huchette2019combinatorial} transfer this problem into a minimum biclique cover problem in the associated conflict graph.

\begin{definition}[biclique covers]
A \textit{biclique} graph is a complete biparitite graph $(A \cup B, A \times B)$, which is denoted as $\{A, B\}$. A \textit{biclique cover} of graph $G = (J, E)$ is a collection of biclique subgraphs of $G$ that covers the edge set $E$.
\end{definition}

\begin{proposition}[Theorem 3~\cite{huchette2019combinatorial}] \label{prop:biclique_cover}
If $\{\{A^j, B^j\}\}_{j=1}^t$ is a biclique cover of the conflict graph $G^c_{\mathcal{S}}$ for a pairwise IB-representable $\CDC(\mathcal{S})$, then a pairwise IB-scheme of $\CDC(\mathcal{S})$ is given by $L^j = J \setminus A^j, R^j = J \setminus B^j, \forall j \in \llbracket t\rrbracket$. On the other hand, if $\{\{L^j, R^j\}\}_{j=1}^t$ is a pairwise IB-scheme for $\CDC(\mathcal{S})$, then a biclique cover of the conflict graph $G^c_{\mathcal{S}}$ is given by $A^j = J \setminus L^j, B^j = J \setminus R^j, \forall j \in \llbracket t \rrbracket.$
\end{proposition}

\begin{corollary} \label{cor:cdc_bc}
Given a biclique cover $\{\{A^j, B^j\}\}_{j=1}^t$ of the conflict graph $G^c_{\mathcal{S}}$ for a pairwise IB-representable $\CDC(\mathcal{S})$, the following is an ideal formulation for $\CDC(\mathcal{S})$ with $J:= \bigcup_{S \in \mathcal{S}} S$:
\begin{subequations}
\begin{alignat}{2}
    & \sum_{v \in A^j} \lambda_v \leq z_j, & \forall j \in \llbracket t\rrbracket\\
    & \sum_{v \in B^j} \lambda_v \leq 1 - z_j, \quad & \forall j \in \llbracket t\rrbracket \\
    & \lambda \in \Delta^J \\
    & z_j \in \{0, 1\}, & \forall j \in \llbracket t\rrbracket .
\end{alignat}
\end{subequations}
\end{corollary}

\section{Junction Trees and Pairwise IB-Representability} \label{sec:junction_tree}

In this section, we introduce a new class of CDCs that is pairwise IB-representable: CDCs admitting junction trees.

\begin{definition}[(complete) intersection graphs] \label{def:intersection}
The \textit{(complete) intersection graph} of $\CDC(\mathcal{S})$ is denoted as $\mathcal{K}_{\mathcal{S}} = (\mathcal{S}, \mathcal{E})$ where $\mathcal{E} = \{S^1S^2: S^1,S^2 \in \mathcal{S}\}$. The \textit{middle set} of the edge $S^1 S^2$ is defined as $\MID(S^1 S^2) := S^1 \cap S^2$, and the \textit{weight} is $w(S^1 S^2) := |\MID(S^1 S^2)|$.
\end{definition}

\begin{definition}[junction trees] \label{def:junction_tree}
A \textit{junction tree} of $\CDC(\mathcal{S})$ is denoted as $\mathcal{T}_{\mathcal{S}} = (\mathcal{S}, \mathcal{E})$, where $\mathcal{T}_{\mathcal{S}}$ is a tree and $\mathcal{E}$ satisfies:
\begin{itemize}
    \item For any $S^1, S^2 \in \mathcal{S}$, the unique path $\mathcal{P}$ between $S^1$ and $S^2$ in $\mathcal{T}_{\mathcal{S}}$ satisfies that $S^1 \cap S^2 \subseteq S$ for any $S \in V(\mathcal{P})$, and $S^1 \cap S^2 \subseteq \MID(e)$ for any $e \in E(\mathcal{P})$.
\end{itemize}

We will say that $\CDC(\mathcal{S})$ admitting a junction tree if such a junction tree exists. We denote that $w(G) := \sum_{e \in E(G)} w(e)$ for a graph $G$ with weights on its edge set. We also define the distance between two spanning tree subgraph $T_1, T_2$ of graph $G$ to be $\dist(T_1, T_2) := |E(T_1) \setminus E(T_2)| \equiv |E(T_2) \setminus E(T_1)|$.

\end{definition}

Note that a junction tree of $\CDC(\mathcal{S})$, $\mathcal{T}_{\mathcal{S}}$, is a spanning tree subgraph of the intersection graph, $\mathcal{K}_{\mathcal{S}}$.

In addition, We show that admitting a junction tree is a sufficient condition for $\CDC(\mathcal{S})$ to be pairwise IB-representable.




\begin{theorem} \label{thm:IB_tree}
If $\CDC(\mathcal{S})$ admits a junction tree $\mathcal{T}_{\mathcal{S}}$, then $\CDC(\mathcal{S})$ is pairwise IB-representable.
\end{theorem}

\begin{proof}
Assume that $\CDC(\mathcal{S})$ is not pairwise IB-representable. By Proposition~\ref{prop:pairwise_IB_at_most_two}, there exists a minimal infeasible set $S \subseteq J$ with cardinality greater than 2. Without loss of generality, assume that $S = \{x_1, x_2, \hdots, x_k\}$ where $k = |S|$ and $k \geq 3$. Since $S$ is a minimal infeasible set, then $S \setminus \{x_1\}$, $S \setminus \{x_2\}$ and $S \setminus \{x_3\}$ are not infeasible sets. Thus, there exists some $S^{x_1}$, $S^{x_2}$, $S^{x_3} \in \mathcal{S}$ such that $(S \setminus \{x_j\}) \subseteq S^{x_j}$. By~\cite[Proposition 4.1]{bondy2008graph}: in a tree, any two vertices are connected by exactly one path, there exists a vertex $S^0$ in $\mathcal{T}_{\mathcal{S}}$ such that it is on the paths between any two of $S^{x_1}$, $S^{x_2}$, $S^{x_3}$. By the existence of $S^0$ and definition of junction trees, $S \setminus \{x_i, x_j\} \subseteq S^0$ for all distinct $i, j \in \llbracket 3 \rrbracket$. Hence, $S \subseteq S^0$, which is a contradiction.
\qed\end{proof}


\section{When a Disjunctive Constraint Admits a Junction Tree}

Theorem~\ref{thm:IB_tree} shows that $\CDC$ admitting a junction tree is pairwise IB-representable. Then, we can apply Proposition~\ref{prop:biclique_cover} to construct a small MIP formulation by solving a minimum biclique cover problem on its conflict graph. We start by studying how to determine whether a CDC admits a junction tree.


\begin{theorem} \label{thm:check_jc_alg}
A $\CDC(\mathcal{S})$ admits a junction tree if and only if each maximum spanning tree $\mathcal{S}$ of the intersection graph $\mathcal{K}_{\mathcal{S}}$ satisfies:
\begin{itemize}
    \item for each $e \in E(\mathcal{T})$ and vertex sets $S_1$ and $S_2$ that form connected subgraphs of $\mathcal{T} \setminus e$, $\bigcup_{S \in \mathcal{S}_1} S \cap \bigcup_{S \in \mathcal{S}_2} S \not\subseteq \MID(e)$.
\end{itemize}
\end{theorem}

\begin{corollary}
Given an arbitrary $\mathcal{S} := \{S^i\}_{i=1}^d$ and $n := |\bigcup_{i=1}^d S^i|$, it is possible to determine whether $\CDC(\mathcal{S})$ admits a junction tree in time polynomial in $d$ and $n$.
\end{corollary}

In order to prove Theorem~\ref{thm:check_jc_alg}, we first show that $\mathcal{T}$ cannot be a junction tree of $\CDC(\mathcal{S})$ if $\mathcal{T}$ is not a maximum spanning tree of $\mathcal{K}_{\mathcal{S}}$ in Proposition~\ref{prop:tree_mst}. Then, we prove that $\CDC(\mathcal{S})$ admits a junction tree if and only if all maximum spanning trees of $\mathcal{K}_{\mathcal{S}}$ are junctions trees of $\CDC(\mathcal{S})$ in Proposition~\ref{prop:mst_junction}. Finally, we prove that a spanning tree $\mathcal{T}$ of $\mathcal{K}_{\mathcal{S}}$ is a junction tree of $\CDC(\mathcal{S})$ if and only if each edge $e$ in $\mathcal{T}$ separates vertices into $\mathcal{S}_1$ and $\mathcal{S}_2$ such that $\left(\bigcup_{S \in \mathcal{S}_1} S \right) \cap \left(\bigcup_{S \in \mathcal{S}_2} S \right) \subseteq \MID(e)$ in Proposition~\ref{prop:check_junction_tree}.

\begin{lemma}[Theorem 3-16~\cite{deo2017graph}] \label{lm:tree_local} A spanning tree $T$ of a weighted graph $G$ is a maximum spanning tree if and only if there does not exist another spanning tree $T'$ of $G$ such that $w(T) < w(T')$ and $\dist(T, T') = 1$.
\end{lemma}

\begin{proposition} \label{prop:tree_mst}
Given $\CDC(\mathcal{S})$ and its intersection graph $\mathcal{K}_{\mathcal{S}}$, any non-maximum spanning tree of $\mathcal{K}_{\mathcal{S}}$ cannot be a junction tree of $\CDC(\mathcal{S})$.
\end{proposition}

\begin{proof}
Let $\mathcal{T}$ be an arbitrary spanning tree subgraph of $\mathcal{K}_{\mathcal{S}}$ that is not maximum spanning tree. By Lemma~\ref{lm:tree_local}, we know that there exists another spanning tree $\mathcal{T}_1$ such that $w(\mathcal{T}_1) > w(\mathcal{T})$ and $\dist(\mathcal{T}_1, \mathcal{T}) = 1$. Thus, let the edge $e_1 \in E(\mathcal{T}_1) \setminus E(\mathcal{T})$ and the endpoints of $e_1$ to be $S^1$ and $S^2$. Since $\mathcal{T}$ is a tree, there exists a unique path subgraph, $\mathcal{P}$, of $\mathcal{T}$ between $S^1$ and $S^2$. Since $\dist(\mathcal{T}_1, \mathcal{T}) = 1$, then the edge $e \in E(\mathcal{T}) \setminus E(\mathcal{T}_1)$ is in $\mathcal{P}$. Since $w(\mathcal{T}_1) > w(\mathcal{T})$, $w(e_1) > w(e)$. Then, there must exist $x \in \MID(e_1) := (S^1 \cap S^2)$ that is not in $\MID(e)$, which implies that $\mathcal{T}$ is not a junction tree of $\CDC(\mathcal{S})$ by Definition~\ref{def:junction_tree}.
\qed\end{proof}

Next, we want to show that $\CDC(\mathcal{S})$ admits a junction tree if and only if all maximum spanning trees of $\mathcal{K}_{\mathcal{S}}$ are junctions trees of $\CDC(\mathcal{S})$. Lemma~\ref{lm:tree_not_junction} states that if $\mathcal{T}$ is a junction tree of $\CDC(\mathcal{S})$, then any maximum spanning tree subgraph of $\mathcal{K}_{\mathcal{S}}$ such that its distance to $\mathcal{T}$ is one is also a junction tree of $\CDC(\mathcal{S})$. Lemma~\ref{lm:seq_mst} states that, given two arbitrary maximum spanning tree subgraphs of $\mathcal{K}_{\mathcal{S}}$, there exists a sequence of maximum spanning tree subgraphs such that the distance between any two consecutive trees is one.

\begin{lemma} \label{lm:tree_not_junction}
Given $\CDC(\mathcal{S})$ and its intersection graph $\mathcal{K}_{\mathcal{S}}$, let $\mathcal{T}_1$ and $\mathcal{T}_2$ be two maximum spanning tree subgraphs of $\mathcal{K}_{\mathcal{S}}$ such that $\dist(\mathcal{T}_1, \mathcal{T}_2) = 1$. If $\mathcal{T}_1$ is a junction tree of $\CDC(\mathcal{S})$, then $\mathcal{T}_2$ is a junction tree of $\CDC(\mathcal{S})$ as well.
\end{lemma}

\begin{proof}
Let $e_1$ and $e_2$ be the edges in $(E(\mathcal{T}_1) \setminus E(\mathcal{T}_2))$ and $(E(\mathcal{T}_2) \setminus E(\mathcal{T}_1))$, respectively. By adding $e_2$ to $\mathcal{T}_1$, we can obtain a fundamental cycle $\mathcal{C}$, which contains both edges $e_1$ and $e_2$. Since $\mathcal{T}_1$ is a junction tree of $\CDC(\mathcal{S})$, then $\MID(e_2) \subseteq \MID(e)$ for any $e \in (E(\mathcal{C}) \setminus \{e_2\})$. Since $\mathcal{T}_2$ is also a maximum spanning tree, then $w(e_2) \geq w(e)$ for any $e \in (E(\mathcal{C}) \setminus \{e_2\})$. (Otherwise, by removing $e_2$ from and adding $e$ to $\mathcal{T}_2$, we can obtain a new spanning tree with larger weights which is a contradiction.) Hence, $\MID(e_2) = \MID(e)$ for any $e \in (E(\mathcal{C}) \setminus \{e_2\})$. Since $e_1 \in \mathcal{C}$, $\MID(e_1) = \MID(e_2) \subseteq S$ for any vertex $S \in \mathcal{C}$.

By removing $e_1$ from $\mathcal{T}_1$, we can partition the vertices into sets $\mathcal{S}_1$ and $\mathcal{S}_2$. For arbitrary distinct vertices $S$ and $S'$ in $\mathcal{S}_1$, since the path between $S$ and $S'$ in $\mathcal{T}_1$ and the path between those two vertices in $\mathcal{T}_2$ are identical, there exists a path between $S$ and $S'$ in $\mathcal{T}_2$ satisfying the second property of the junction tree. It follows the same manner for arbitrary distinct vertices $S$ and $S'$ in $\mathcal{S}_2$. For arbitrary $S \in \mathcal{S}_1$ and $S' \in \mathcal{S}_2$, let $\mathcal{P}_1$ be the unique path between $S$ and $S'$ in $\mathcal{T}_1$. We construct the path $\mathcal{P}_2$ as follows: Let $S^1$ and $S^2$ be the closest vertices to $S$ and $S'$ in $\mathcal{P}_1$ respectively such that $S^1$ and $S^2$ are in $\mathcal{C}$. Then, we can construct a path $\mathcal{P}_2$ from $S$ to $S^1$ along $\mathcal{P}_1$, $S^1$ to $S^2$ along the arc of $\mathcal{C}$ containing $e_2$, and then $S^2$ to $S'$ along $\mathcal{P}_1$. Then, $\mathcal{P}_2$ is the unique path in $\mathcal{T}_2$ between $S$ and $S'$. Then, the edges $E(\mathcal{P}_2) \setminus E(\mathcal{P}_1) \subseteq E(\mathcal{C})$. Thus, $\MID(e) = \MID(e_1)$ for any $e \in E(\mathcal{P}_2) \setminus E(\mathcal{P}_1)$. We can complete the proof since $e_1 \in E(\mathcal{P}_1)$ and $(S \cap S') \subseteq \MID(e_1)$.
\qed\end{proof}

\begin{lemma}[Corollary 1~\cite{kravitz2007two}]\label{lm:mst_exchange}
Take two distinct maximum spanning trees $S$ and $T$ of a connected graph $G$. Let $s \in E(S) \setminus E(T)$. Then, there exists an edge $t \in E(T) \setminus E(S)$ such that both $S + t - s$ and $T + s - t$ are maximum spanning trees.
\end{lemma}

\begin{lemma}\label{lm:seq_mst}
Given a simple graph $G$ and two arbitrary maximum spanning tree subgraphs: $T$ and $T'$, there exists a sequence of maximum spanning tree subgraphs $\{T_i\}_{i=0}^k$ such that $T_0=T$, $T_k=T'$, and $\dist(T_{i-1}, T_i) = 1$ for all $i \in \llbracket k\rrbracket $.
\end{lemma}

\begin{proof}
It is a direct result from Lemma~\ref{lm:mst_exchange}.
\qed\end{proof}

Note that $k$ might be 0, in which case the statement in Lemma~\ref{lm:seq_mst} holds trivially.

\begin{proposition} \label{prop:mst_junction}
Given $\CDC(\mathcal{S})$ and its intersection graph $\mathcal{K}_{\mathcal{S}}$, $\CDC(\mathcal{S})$ admits a junction tree if and only if all maximum spanning trees of $\mathcal{K}_{\mathcal{S}}$ are junctions trees of $\CDC(\mathcal{S})$.
\end{proposition}

\begin{proof}
The backward direction is trivial. If all maximum spanning trees of $\mathcal{K}_{\mathcal{S}}$ are junction trees, then $\CDC(\mathcal{S})$ admits a junction tree.

For the forward direction, suppose that $\CDC(\mathcal{S})$ admits a junction tree. Then, by Proposition~\ref{prop:tree_mst}, we know that there exists a maximum spanning tree, $\mathcal{T}$, of $\mathcal{K}_{\mathcal{S}}$ that is a junction tree. Let $\mathcal{T}'$ be an arbitrary maximum spanning tree of $\mathcal{K}_{\mathcal{S}}$. By Lemma~\ref{lm:seq_mst}, there exists a sequence of maximum spanning tree subgraphs $\{\mathcal{T}_i\}_{i=0}^k$ such that $\mathcal{T}_0=\mathcal{T}$, $\mathcal{T}_k=\mathcal{T}'$, and $\dist(\mathcal{T}_{i-1}, \mathcal{T}_i) = 1$ for all $i \in \llbracket k\rrbracket $. Since $\mathcal{T}$ is a junction tree, $\mathcal{T}_i$ is a junction tree for all $i \in \llbracket k\rrbracket $ by Lemma~\ref{lm:tree_not_junction}. Hence, $\mathcal{T}'$ is a junction tree of $\CDC(\mathcal{S})$ and all maximum spanning trees of $\mathcal{K}_{\mathcal{S}}$ are junction trees of $\CDC(\mathcal{S})$.
\qed\end{proof}

\begin{proposition}\label{prop:check_junction_tree}
A spanning tree $\mathcal{T}$ of $\mathcal{K}_{\mathcal{S}}$ is a junction tree of $\CDC(\mathcal{S})$ if and only if each edge $e$ in $\mathcal{T}$ separates vertices into $\mathcal{S}_1$ and $\mathcal{S}_2$ such that $\left(\bigcup_{S \in \mathcal{S}_1} S \right) \cap \left(\bigcup_{S \in \mathcal{S}_2} S \right) \subseteq \MID(e)$.
\end{proposition}

\begin{proof}
In the forward direction, let $\mathcal{T}$ be a junction tree. Given an edge $e$ in $\mathcal{T}$, the edge $e$ can separate the vertices into $\mathcal{S}_1$ and $\mathcal{S}_2$. For any $S^1 \in \mathcal{S}_1$ and $S^2 \in \mathcal{S}_2$, the edge $e$ is on the unique path between $S^1$ and $S^2$ in $\mathcal{T}$. Thus, by the definition of junction tree in Definition~\ref{def:junction_tree}, $S^1 \cap S^2 \subseteq \MID(e)$. Hence, $\left(\bigcup_{S \in \mathcal{S}_1} S\right) \cap \left(\bigcup_{S \in \mathcal{S}_2} S\right) \subseteq \MID(e)$ for any edge $e$ in $\mathcal{T}$.

In the backward direction, we prove by the contrapositive statement. Assume that $\mathcal{T}$ is not a junction tree of $\CDC(\mathcal{S})$. Then, there exists $S^1, S^2 \in \mathcal{S}$ and an edge $e$ on the path between them such that $S^1 \cap S^2 \not\subseteq \MID(e)$. Since $\mathcal{T}_\mathcal{S}$ is a tree, then $e$ can separate the vertices into $\mathcal{S}_1$ and $\mathcal{S}_2$ where $S^1 \in \mathcal{S}_1$ and $S^2 \in \mathcal{S}_2$. Hence, $\left(\bigcup_{S \in \mathcal{S}_1} S\right) \cap \left(\bigcup_{S \in \mathcal{S}_2} S\right) \not\subseteq \MID(e)$.
\qed\end{proof}

\section{A Heuristic for Biclique Covers of Combinatorial Disjunctive Constraints with Junction Trees}

In this section, we introduce a heuristic designed to produce small biclique covers of $G^c_{\mathcal{S}}$. The algorithm is composed of two parts: separation and merging. We focus on $\CDC(\mathcal{S})$ with a given junction tree $\mathcal{T}_{\mathcal{S}}$. In this way, $\CDC(\mathcal{S})$ is guaranteed to be pairwise IB-representable. The motivation for designing a ``divide and conquer" algorithm is the simple observation that any biclique can naturally divide the minimum biclique cover problem into two subproblems, as shown in Proposition~\ref{prop:biclique_sep}.

\begin{proposition} \label{prop:biclique_sep}
Given a graph $G := (V, \mathcal{E})$, let $\{A, B\}$ be an arbitrary biclique subgraph of $G$. Then, $\mathcal{E} = E(\{A, B\}) \cup E(G_A) \cup E(G_B)$, where $G_A := G(V \setminus A)$ and $G_B := G(V \setminus B)$. 
\end{proposition}

\begin{proof}
Since $\{A, B\}$, $G_A$, $G_B$ are all subgraphs of $G$, then $E(\{A, B\}) \cup E(G_A) \cup E(G_B) \subseteq \mathcal{E}$.

Given an arbitrary $uv \in \mathcal{E}$. If $u, v \in (V \setminus B)$, $uv \in E(G_B)$. Similarly, if $u, v \in (V \setminus A)$, $uv \in E(G_A)$. If $u,v$ are in $A$ and $B$, respectively, then $uv \in E(\{A, B\})$. Thus, $\mathcal{E} \subseteq E(\{A, B\}) \cup E(G_A) \cup E(G_B)$.
\qed\end{proof}

Then, we can show that any edge in the junction tree can help us to find a biclique in the conflict graph of $\CDC(\mathcal{S})$.

\begin{proposition} \label{prop:separation_tree}
Given $\CDC(\mathcal{S})$ and one of its junction trees $\mathcal{T}_{\mathcal{S}}$, any edge $e$ of $\mathcal{T}_{\mathcal{S}}$ can partition $\mathcal{S}$ into $\mathcal{S}_1$ and $\mathcal{S}_2$ such that $\{\bigcup_{S \in \mathcal{S}_1} S \setminus \MID(e), \bigcup_{S \in \mathcal{S}_2} S \setminus \MID(e) \}$ is a biclique of the conflict graph $G^c_{\mathcal{S}}$.
\end{proposition}

\begin{proof}
Since every edge in a tree is a cut, $e$ can partition $\mathcal{S}$ into two sets of vertices, $\mathcal{S}_1$ and $\mathcal{S}_2$, in $\mathcal{T}_{\mathcal{S}}$. Given an arbitrary $u \in \bigcup_{S \in \mathcal{S}_1} S \setminus \MID(e)$ and $v \in \bigcup_{S \in \mathcal{S}_2} S \setminus \MID(e)$, we want to prove that $uv$ is an edge in the conflict graph $G^c_{\mathcal{S}}$. Equivalently, we need to prove that $\{u, v\} \not\subseteq S$ for any $S \in \mathcal{S}$.

We know that $v \in S^2$ for some $S^2 \in \mathcal{S}_2$ and $v \not\in \MID(e)$. Given an arbitrary $S^1 \in \mathcal{S}_1$, since there exists only one path between $S^1$ and $S^2$ in the junction tree $\mathcal{T}_{\mathcal{S}}$, then $(S^1 \cap S^2) \subseteq \MID(e)$. Since $v \not\in \MID(e)$, then $v \not\in S^1$. Thus, $v \not\in \bigcup_{S \in \mathcal{S}_1} S$. Similarly, $u \not\in \bigcup_{S \in \mathcal{S}_2} S$, which implies that there does not exist $S \in \mathcal{S}$ such that $\{u, v\} \in S$.
\qed\end{proof}

In Proposition~\ref{prop:sep0} and Proposition~\ref{prop:separation}, we show that we only need to do the ``divide and conquer" on the junction tree to find a biclique cover of the conflict graph of $\CDC(\mathcal{S})$, which leads us to a separation subroutine in Algorithm~\ref{alg:biclique_sep}.

\begin{proposition} \label{prop:sep0}
Given $\CDC(\mathcal{S})$ and one of its junction trees $\mathcal{T}_{\mathcal{S}}$, any edge $e$ of $\mathcal{T}_{\mathcal{S}}$ can partition vertices of $\mathcal{T}_{\mathcal{S}}$ into $\mathcal{S}_1$ and $\mathcal{S}_2$. Let $G^c_{\mathcal{S}_1}$ be the conflict graph of $\CDC(\mathcal{S}_1)$ and $G^c_{\mathcal{S}_2}$ be the conflict graph of $\CDC(\mathcal{S}_2)$. Then, the union of $\left\{\bigcup_{S \in \mathcal{S}_1} S \setminus \MID(e), \bigcup_{S \in \mathcal{S}_2} S \setminus \MID(e)\right\}$ and the biclique covers of $G^c_{\mathcal{S}_1}$ and $G^c_{\mathcal{S}_2}$ is a biclique cover of $G^c_{\mathcal{S}}$.
\end{proposition}

\begin{proof}
Since $\mathcal{T}_S$ is a tree, then each edge $e$ of $\mathcal{T}_S$ is an edge cut to partition the vertices into $\mathcal{S}_1$ and $\mathcal{S}_2$. By Proposition~\ref{prop:separation_tree}, we know that $\{A:= \bigcup_{S \in \mathcal{S}_1} S \setminus \MID(e), B:= \bigcup_{S \in \mathcal{S}_2} S \setminus \MID(e) \}$ is a biclique subgraph of conflict graph $G^c_{\mathcal{S}}$. Let induced subgraphs $G^c_A := G^c_{\mathcal{S}}(V(G^c_{\mathcal{S}}) \setminus A)$ and $G^c_B := G^c_{\mathcal{S}}(V(G^c_{\mathcal{S}}) \setminus B)$. By Proposition~\ref{prop:biclique_sep}, we know that $E(G^c_{\mathcal{S}}) = E(\{A, B\}) \cup E(G^c_A) \cup E(G^c_B)$. Hence, the union of $\{A, B\}$ and any biclique covers of $G^c_A$ and $G^c_B$ is a biclique cover of $G^c_{\mathcal{S}}$.

The rest of the proof will show that $G^c_A$ is the conflict graph of $\CDC(\mathcal{S}_2)$, $G^c_{\mathcal{S}_2}$. The proof that $G^c_B$ is the conflict graph of $\CDC(\mathcal{S}_1)$ follows the same manner.

Denote the vertices incident to $e$ to be $S^1 \in \mathcal{S}_1$ and $S^2 \in \mathcal{S}_2$. Then, $S^1 \cap S^2 = \MID(e)$ by the definition of junction tree of $\CDC(\mathcal{S})$. Also, $(\bigcup_{S \in \mathcal{S}_1} S) \cap (\bigcup_{S \in \mathcal{S}_2} S) = \MID(e)$ because $\mathcal{T}_{\mathcal{S}}$ is a tree. Hence, $V(G^c_{\mathcal{S}}) \setminus A = \bigcup_{S \in \mathcal{S}_2} S$ and $V(G^c_A) = V(G^c_{\mathcal{S}_2})$.

Since $G^c_A$ is a subgraph of $G^c_{\mathcal{S}}$, an edge $uv$ is in $G^c_A$ if and only if $\{u,v\} \not\in S$ for all $S \in \mathcal{S}$. Similarly, an edge $uv$ is in $G^c_{\mathcal{S}_2}$ if and only if $\{u,v\} \not\subseteq S$ for all $S \in \mathcal{S}_2$. Thus, it is not hard to see that $E(G^c_A) \subseteq E(G^c_{\mathcal{S}_2})$.

In the other direction, we know that $\{u,v\} \not\subseteq S$ for all $S \in \mathcal{S}_2$. We want to show that $\{u,v\} \not\subseteq S$ for all $S \in \mathcal{S}$ by contradiction. Assume that $\{u, v\} \in S^0$ for some $S^0 \in \mathcal{S}_1$. Since $u,v \in \bigcup_{S \in \mathcal{S}_2} S$, $u$ and $v$ must be in $\MID(e)$ by the definition of junction tree of $\CDC(\mathcal{S})$. By the definition of $\MID(e)$, the vertex incident to $e$ in $\mathcal{S}_2$ contains both $u$ and $v$, which is a contradiction. Hence, $ E(G^c_{\mathcal{S}_2}) \subseteq E(G^c_A)$.
\qed\end{proof}

\begin{proposition}\label{prop:separation}
Given $\CDC(\mathcal{S})$ and its junction tree $\mathcal{T}_{\mathcal{S}}$, any edge $e$ of $\mathcal{T}_{\mathcal{S}}$ can partition the tree into two subtrees $\mathcal{T}_1$ and $\mathcal{T}_2$ whose vertices are $\mathcal{S}_1$ and $\mathcal{S}_2$, respectively. Then, $\mathcal{T}_1$ is a junction tree of $\CDC(\mathcal{S}_1)$ and $\mathcal{T}_2$ is a junction tree of $\CDC(\mathcal{S}_2)$.
\end{proposition}

\begin{proof}
Given an arbitrary edge $e$ in $\mathcal{T}_{\mathcal{S}}$, since $\mathcal{T}_{\mathcal{S}}$ is a tree, edge $e$ can partition the graph into two subtrees $\mathcal{T}_1$ and $\mathcal{T}_2$ where the vertices of two subtrees are $\mathcal{S}_1$ and $\mathcal{S}_2$, respectively. It is sufficient to prove that $\mathcal{T}_1$ is a junction tree of $\CDC(\mathcal{S}_1)$. The proof that $\mathcal{T}_2$ is a junction tree of $\CDC(\mathcal{S}_2)$ follows the same manner.

For arbitrary distinct $S, S' \in \mathcal{S}_1$, there exists a unique path $P$ between $S$ and $S'$ in $\mathcal{T}_{\mathcal{S}}$. Since $\mathcal{T}_{\mathcal{S}}$ is a junction tree of $\CDC(\mathcal{S})$, $S \cap S' \subseteq \MID(e')$ for any $e' \in E(P)$ and $S \cap S' \subseteq S''$ for any $S'' \in V(P)$. Because $\mathcal{T}_1$ is a subtree of $\mathcal{T}_{\mathcal{S}}$, $P$ is also a path in $\mathcal{T}_1$. Hence, according to Definition~\ref{def:junction_tree} (junction tree), $\mathcal{T}_1$ is a junction tree of $\CDC(\mathcal{S}_1)$.
\qed\end{proof}


\begin{theorem} \label{thm:worst_case}
Suppose that the input of Algorithm~\ref{alg:biclique_sep} is a junction tree of $\CDC(\mathcal{S})$, $\mathcal{T}_{\mathcal{S}}$. Then, the output $\bc = \Call{Separation}{\mathcal{T}_{\mathcal{S}}}$ is a biclique cover of the conflict graph $G^c_{\mathcal{S}}$ with a size of at most $|\mathcal{S}| - 1$.
\end{theorem}

\begin{proof}
If the junction tree only has one vertex, then there is not any edge in the corresponding conflict graph. Thus, we do not need to add any biclique to the result. Implied by Proposition~\ref{prop:sep0} and Proposition~\ref{prop:separation}, we know that the output $\bc$ is a biclique cover of $G^c_{\mathcal{S}}$. There are $|\mathcal{S}|-1$ edges in the junction tree $\mathcal{T}_{\mathcal{S}}$ and we only add one biclique into the result when removing each edge in the junction tree. Thus, the size of $\bc$ is at most $|\mathcal{S}| - 1$.
\qed\end{proof}

\begin{algorithm}[H]
\begin{algorithmic}[1]
\State \textbf{Input}: A junction tree $\mathcal{T}_{\mathcal{S}}$ of $\CDC(\mathcal{S})$.
\State \textbf{Output}: A tuple of bicliques $\bc := (\{A^{k}, B^{k}\})_k$ of the conflict graph $G^c_{\mathcal{S}}$ of $\CDC(\mathcal{S})$.
\Function{Separation}{$\mathcal{T}_{\mathcal{S}}$}
\If{$|V(\mathcal{T}_{\mathcal{S}})| \leq 1$}
\State \textbf{return} $()$.
\EndIf
\State Find an edge $e$ to cut $\mathcal{T}_{\mathcal{S}}$ into two components $\mathcal{T}_1$ and $\mathcal{T}_2$ such that $|V(\mathcal{T}_1) - V(\mathcal{T}_2)|$ is minimized. \label{ln:e_selection} \Comment{It can be done by iterating all edges in $\mathcal{T}_{\mathcal{S}}$.}\label{ln:bc_sep_edge}
\State $A := \bigcup_{S \in V(\mathcal{T}_1)} S \setminus \MID(e)$; $B := \bigcup_{S \in V(\mathcal{T}_2)} S \setminus \MID(e)$. \label{ln:bc_sep_bc}
\State $\bc_1 := \Call{Separation}{\mathcal{T}_1}$; $\bc_2 := \Call{Separation}{\mathcal{T}_2}$.
\State Pop the first element of $\bc_1$ as $b_1$ and pop the first element of $\bc_2$ as $b_2$. \Comment{If a set is $()$, the biclique popped is also $()$.}
\State \textbf{return} $(\{A, B\}, b_1, b_2, \bc_1 \hdots, \bc_2 \hdots)$. \Comment{Store the biclique by level order.} \label{ln:bc_return}
\EndFunction
\end{algorithmic}
\caption{A separation subroutine.} \label{alg:biclique_sep}
\end{algorithm}

Note that the decision of selecting edge $e$ in Line~\ref{ln:e_selection} of Algorithm~\ref{alg:biclique_sep} was made due to simplicity, but it might not be the best possible choice; we leave that tuning to future work. Also note that $\bc_i \hdots$ represents all the elements in $\bc_i$ for $i \in \{1,2\}$ in Line~\ref{ln:bc_return} of Algorithm~\ref{alg:biclique_sep}.

\begin{algorithm}[H]
\begin{algorithmic}[1]
\State \textbf{Input}: A junction tree $\mathcal{T}_{\mathcal{S}}$ of $\CDC(\mathcal{S})$ and conflict graph $\mathcal{G}^c_{\mathcal{S}}$ of $\CDC(\mathcal{S})$.
\State \textbf{Output}: A biclique cover $\{\{A^k, B^k\}\}_k$ of the conflict graph $G^c_{\mathcal{S}}$.
\State $\bc :=\Call{Separation}{\mathcal{T}_{\mathcal{S}}}$.
\State $\bc_{\merged} := ()$.
\For{$k \in \llbracket |\bc| \rrbracket$} \label{ln:bc_start_for}
\State $\flag := \True$ and $\{A, B\} := \bc[k]$.
\For{$\{A', B'\} \in \bc_{\merged}$}
\If {$\{A' \cup A, B' \cup B\}$ is a biclique subgraph of $G^c_{\mathcal{S}}$}
\State Replace $\{A', B'\}$ with $\{A' \cup A, B' \cup B\}$ in $\bc_{\merged}$ and set $\flag = \False$.
\State \textbf{break}
\ElsIf {$\{A' \cup B, B' \cup A\}$ is a biclique subgraph of $G^c_{\mathcal{S}}$}
\State Replace $\{A', B'\}$ with $\{A' \cup B, B' \cup A\}$ in $\bc_{\merged}$ and set $\flag = \False$.
\State \textbf{break}
\EndIf
\EndFor
\If {$\flag$}
\State $\bc_{\merged} := (\bc_{\merged} \hdots, \{A, B\})$.
\EndIf
\EndFor \label{ln:bc_end_for}
\State \textbf{return} $\set(\bc_{\merged})$ \Comment $\set(\bc_{\merged})$ converts the tuple to a set.
\end{algorithmic}
\caption{Find a small biclique cover of $G^c_{\mathcal{S}}$ in polynomial time.} \label{alg:biclique}
\end{algorithm}


\begin{theorem} \label{thm:alg_sep_bound_time}
Suppose $d := |\mathcal{S}|$, $n := |V(G^c_{\mathcal{S}})| \equiv |\bigcup_{S \in \mathcal{S}} S|$, and $m: = |E(G^c_{\mathcal{S}})|$, Algorithm~\ref{alg:biclique} will return a biclique cover of the conflict graph $G^c_{\mathcal{S}}$ of $\CDC(\mathcal{S})$ with size at most $d-1$ in $O(d^3 + n d^2 + m d^2)$ time.
\end{theorem}
\begin{proof}
Algorithm~\ref{alg:biclique_sep} returns a biclique cover of the conflict graph $G^c_{\mathcal{S}}$. The operations from lines \ref{ln:bc_start_for} to \ref{ln:bc_end_for} of Algorithm~\ref{alg:biclique} only merge bicliques. Hence, Algorithm~\ref{alg:biclique} also returns a biclique cover with size at most $d-1$. 

Since each $\Call{Separation}{}$ function removes an edge in $\mathcal{T}_{\mathcal{S}}$ and there are only $d-1$ edges in $\mathcal{T}_{\mathcal{S}}$, $\Call{Separation}{}$ function is called at most $d-1$ times. Creating bicliques in line~\ref{ln:bc_sep_bc} of Algorithm~\ref{alg:biclique_sep} requires $O(d n)$ time. There are at most $d-1$ edges and checking the numbers of vertices in two subtrees by each edge cut can be done in $O(d)$ time. Thus, the computational time for $\Call{Separation}{\mathcal{G}_{\mathcal{S}}}$ is $O(d^2 n + d^3)$.

As proven in Theorem~\ref{thm:worst_case}, the number of bicliques returned by $\Call{Separation}{\mathcal{T}_\mathcal{S}}$ is at most $d-1$. Then, the number of iterations from lines \ref{ln:bc_start_for} to \ref{ln:bc_end_for} of Algorithm~\ref{alg:biclique} is $O(d^2)$. Furthermore, merging two bicliques requires $O(n)$ time but checking whether a biclique is a subgraph of $G^c_{\mathcal{S}}$ requires $O(m)$ time. The total operation time from lines \ref{ln:bc_start_for} to \ref{ln:bc_end_for} of Algorithm~\ref{alg:biclique} is $O(m d^2)$ time.
\qed\end{proof}

\section{Combinatorial Disjunctive Constraints of Generalized Special Ordered Sets} \label{sec:sosk_cdc}

In this section, we first formally define the generalized special ordered sets $\SOS k(N)$. Each $\SOS k(N)$ contains $N$ non-negative variables such that each variable has a unique order from $1$ to $N$, at most $k$ consecutive variables can be nonzero, and the total sum of $N$ variables is 1. 

\begin{definition}[generalized special ordered sets] \label{def:sosk} 
A \textit{generalized special ordered set constraint} with two positive integers $k \leq N$:
\begin{align} \label{eq:sosk}
    \lambda \in \SOS k(N) &:= \bigcup_{i=1}^{N-k+1} \left\{\lambda \in \Delta^{N}: \lambda_j = 0, \forall j \in \llbracket N\rrbracket  \setminus \llbracket i, i+k-1\rrbracket  \right\}.
\end{align}
\end{definition}

Huchette and Vielma~\cite{huchette2019combinatorial} used the graph union construction to produce an IB-scheme for $\SOS k(N)$ of depth $\lceil \log_2(\lceil N / k \rceil - 1) \rceil + 3k$ with two positive integers $N > k$. We derive a better bound by apply Algorithm~\ref{alg:biclique}. Denote that $\diameter(S) = \max\{|s_1 - s_2|: \forall s_1, s_2 \in S\}$ and $\dist(A, B) = \min\{|u - v|: \forall u \in A, v \in B\}$. Note that the conflict graph of $\SOS k(N)$ is a graph with no edges when $N = k$, so its biclique cover number is 0. Also note that the conflict graph of $\SOS 1(N)$ is a complete graph with $N$ vertices. Its biclique cover number is $\lceil \log_2(N) \rceil$ and the construction of such a biclique cover is simple. Hence we can focus on the case that $N > k \geq 2$ without loss of generality. We first show that $\SOS k(N)$ admits a junction tree.


\begin{proposition} \label{prop:SOSkN_tree}
Given two positive integers $N > k$, then the combinatorial disjunctive constraint of $\SOS k(N)$, $\CDC(\mathcal{S})$ where $\mathcal{S} = \{\llbracket i, i+k-1\rrbracket \}_{i=1}^{N-k+1}$, admits a junction tree. Therefore, $\SOS k(N)$ is pairwise IB-representable.
\end{proposition}

\begin{proof}
Let $S^i = \llbracket i, i+k-1\rrbracket $. It is not hard to see that we can construct a junction tree $\mathcal{T}_{\mathcal{S}} := \{\mathcal{S}, \mathcal{E}\}$, where $\mathcal{E} = \{S^i S^{i+1}: \forall i \in \llbracket 1, N-k\rrbracket \}$. The graph is visualized in Figure~\ref{fig:SOSkN}. By Theorem~\ref{thm:IB_tree}, $\SOS k(N)$ is pairwise IB-representable.
\qed\end{proof}

\begin{figure}[H]
\centering
\begin{tikzpicture}[scale=0.2]
\svertex(1) [-{Latex[length=2.5mm]}, minimum size = 30pt, line width=0.5mm, draw=black] at (0,0) [] {$S^1$};
\svertex(2) [-{Latex[length=2.5mm]}, minimum size = 30pt, line width=0.5mm, draw=black] at (12, 0) [] {$S^2$};
\svertex(3) [-{Latex[length=2.5mm]}, minimum size = 30pt, line width=0.5mm, draw=white] at (24, 0) [] {$\hdots$};
\svertex(4) [-{Latex[length=2.5mm]}, minimum size = 30pt, line width=0.5mm, draw=black] at (36, 0) [] {$S^{N-k}$};
\svertex(5) [-{Latex[length=2.5mm]}, minimum size = 30pt, line width=0.5mm, draw=black] at (48, 0) [] {$S^{N - k + 1}$};
\draw [-] [line width=0.5mm] (1) to node {} (2);
\draw [-] [line width=0.5mm] (2) to node {} (3);
\draw [-] [line width=0.5mm] (3) to node {} (4);
\draw [-] [line width=0.5mm] (4) to node {} (5);
\end{tikzpicture}
\caption{A junction tree of $\SOS k (N)$, where each $S^i = \llbracket i, i+k-1\rrbracket $ and $\MID(S^iS^{i+1}) = \llbracket i+1, i+k-1\rrbracket $.}
\label{fig:SOSkN}
\end{figure}
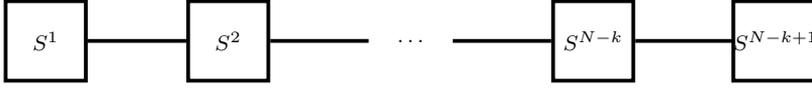

\begin{proposition} \label{prop:larger_SOS}
Given three positive integers $N_1 \geq N_2 > k$, let $J_1 = \llbracket N_1\rrbracket $ and $J_2 = \llbracket N_2\rrbracket $. Take $\mathcal{S}_1$ and $\mathcal{S}_2$ to correspond to $\SOS k(N_1)$ and $\SOS k(N_2)$, respectively. Then, the conflict graph $G^c_{\mathcal{S}_2}$ of $\SOS k(N_2)$ is an induced subgraph of the conflict graph $G^c_{\mathcal{S}_1}$ of $\SOS k(N_1)$. In other words, for any distinct vertices $u, v \in J_2$, $uv \in E(G^c_{\mathcal{S}_1})$ if and only if $uv \in E(G^c_{\mathcal{S}_2})$.
\end{proposition}

\begin{proof}
In the forward direction, suppose that $uv \in G^c_{\mathcal{S}_1}$. Then, $\{u, v\} \not\subseteq S$ for any $S \in \mathcal{S}_1$. Since $\mathcal{S}_2 \subseteq \mathcal{S}_1$, then $\{u, v\} \not\subseteq S$ for all $S \in \mathcal{S}_2$. Hence, $uv \in G^c_{\mathcal{S}_2}$.

In the backward direction, suppose that $uv \in G^c_{\mathcal{S}_2}$. By the definition of $\SOS k(N)$, $\mathcal{S}_2 = \{\llbracket i, i+k-1\rrbracket\}_{i=1}^{N_2-k+1}$. Thus, $\{u, v\} \not\subseteq S$ for all $S \in \mathcal{S}_1$ because $\diameter(S) < k$ for all $S \in \mathcal{S}_1$ and $|u - v| \geq k$.
\qed\end{proof}

By Proposition~\ref{prop:larger_SOS}, we know that if $N_1 \geq N_2$, then the depth of the biclique cover of the conflict graph of $\SOS k(N_2)$ is no larger than the depth of the biclique cover of the conflict graph of $\SOS k(N_1)$: given a biclique cover of the conflict graph of $\SOS k(N_1)$, we can construct a biclique cover of the conflict graph of $\SOS k(N_1)$ by only keep the vertices in $\llbracket N_2\rrbracket$ of all the bicliques. Then, it is sufficient for us to focus on the case where $N = 2^b + k - 1$, where $b$ is a positive integer.

In Proposition~\ref{prop:sosk_ub}, we construct a biclique cover of $\SOS k(N)$ for the conflict graph of $\SOS k(N)$ by the separation subroutine in Algorithm~\ref{alg:biclique_sep}.

\begin{proposition} \label{prop:sosk_ub}
Assume that $N = 2^b + k - 1$, $k \geq 2$, and $b$ are positive integers. Let $J = \llbracket N\rrbracket $ and take $\mathcal{S}$ to correspond to the $\SOS k(N)$ constraint on $J$, i.e. $\mathcal{S} := \{\llbracket i, i+k-1\rrbracket \}_{i=1}^{2^b}$, and let $G^c_{\mathcal{S}}$ be the associated conflict graph. Then, $\{\{A^{i, j}, B^{i, j}\}: i \in \llbracket 0, b-1\rrbracket, j \in \llbracket 0, 2^i-1\rrbracket\}$ is a biclique cover of $G^c_{\mathcal{S}}$, where
\begin{subequations}
\begin{align}
    A^{i,j} &:= \llbracket 1 + j 2^{b-i}, (2j + 1) 2^{b - i - 1}\rrbracket , \\
    B^{i, j} &:= \llbracket (2j+1) 2^{b-i-1} +k, (j+1) 2^{b-i} + k-1\rrbracket .
\end{align}
\end{subequations}
\end{proposition}

\begin{proof}
Denote that $S^i = \llbracket i, i + k - 1\rrbracket $, then $\mathcal{S} = \{S^i\}_{i=1}^{2^b}$. By Proposition~\ref{prop:SOSkN_tree}, we know that a junction tree of $\mathcal{S}$, $\mathcal{T}_{\mathcal{S}} = (\mathcal{S}, \mathcal{E})$ where $\mathcal{E} = \{S^i S^{i+1}, \forall i \in \llbracket N-k\rrbracket \}$. By calling $\Call{Separation}{\mathcal{T}_{\mathcal{S}}}$ in Algorithm~\ref{alg:biclique_sep}, we can obtain a biclique cover with the size of $(2^b - 1)$. As shown in Figure~\ref{fig:sosk_visual}, we denote $A_{\pre}^{i, j}$ and $B_{\pre}^{i, j}$ for $i \in \{0, \hdots, b-1\}$ and $j \in \{0, \hdots, 2^i-1\}$ as
\begin{subequations}
\begin{align}
    A_{\pre}^{i, j} &:= \bigcup_{l=1 + j 2^{b-i}}^{j 2^{b-i} + 2^{b-i-1}} S^l = \llbracket 1 + j 2^{b-i}, 2^{b - i - 1} + j 2^{b - i} + k-1\rrbracket , \\
    B_{\pre}^{i, j} &:= \bigcup_{l=j 2^{b-i} + 2^{b-i-1}+1}^{(j+1) 2^{b-i}} S^l= \llbracket j 2^{b-i} + 2^{b-i-1}+1, (j+1) 2^{b-i} + k-1\rrbracket .
\end{align}
\end{subequations}

Each biclique $\{A^{i, j}, B^{i, j}\}$ for $i \in \{0, \hdots, b-1\}$ and $j \in \{0, \hdots, 2^i-1\}$ can be expressed as 
\begin{subequations}
\begin{align}
    A^{i,j} &:= A_{\pre}^{i, j} \setminus B_{\pre}^{i, j} = \llbracket 1 + j 2^{b-i}, (2j + 1) 2^{b - i - 1}\rrbracket , \\
    B^{i, j} &:= B_{\pre}^{i, j} \setminus A_{\pre}^{i, j} = \llbracket (2j+1) 2^{b-i-1} +k, (j+1) 2^{b-i} + k-1\rrbracket .
\end{align}
\end{subequations}
\qed\end{proof}

We then show that many bicliques in the biclique cover $\{\{A^{i, j}, B^{i, j}\}: i \in \{0, \hdots, b-1\}, j \in \{0, \hdots, 2^i-1\}\}$ in Proposition~\ref{prop:sosk_ub} can be merged into a larger biclique subgraph of the conflict graph of $\SOS k (N)$. Hence, we can obtain a biclique cover with a smaller size.

\begin{figure}[H]
\centering
\resizebox{\textwidth}{!}{
\begin{tikzpicture}[scale=0.14]
\svertex(1) [-{Latex[length=2.5mm]}, minimum size = 30pt, line width=0.5mm, draw=black] at (0,0) [] {$1$};
\svertex(2) [-{Latex[length=2.5mm]}, minimum size = 30pt, line width=0.5mm, draw=white] at (12, 0) [] {$\hdots$};
\svertex(3) [-{Latex[length=2.5mm]}, minimum size = 30pt, line width=0.5mm, draw=black] at (24, 0) [] {${2^{b-1}}$};
\svertex(4) [-{Latex[length=2.5mm]}, minimum size = 30pt, line width=0.5mm, draw=black] at (36, 0) [] {${2^{b-1}+1}$};
\svertex(5) [-{Latex[length=2.5mm]}, minimum size = 30pt, line width=0.5mm, draw=white] at (48, 0) [] {$\hdots$};
\svertex(6) [-{Latex[length=2.5mm]}, minimum size = 30pt, line width=0.5mm, draw=black] at (60, 0) [] {${2^{b}}$};
\draw [-] [line width=0.5mm] (1) to node {} (2);
\draw [-] [line width=0.5mm] (2) to node {} (3);
\draw [-] [line width=0.5mm, dashed] (3) to node {} (4);
\draw [-] [line width=0.5mm] (4) to node {} (5);
\draw [-] [line width=0.5mm] (5) to node {} (6);
\draw[rotate=0,red,thick,dashed] (12, 0) ellipse (17 and 8) (12, 10) node{$A_{\pre}^{0,0}$};
\draw[rotate=0, blue,thick,dashed] (48, 0) ellipse (17 and 8) (48, 10) node{$B_{\pre}^{0,0}$};

\svertex(7) [-{Latex[length=2.5mm]}, minimum size = 30pt, line width=0.5mm, draw=black] at (-36, -20) [] {$1$};
\svertex(8) [-{Latex[length=2.5mm]}, minimum size = 30pt, line width=0.5mm, draw=white] at (-24, -20) [] {$\hdots$};
\svertex(9) [-{Latex[length=2.5mm]}, minimum size = 30pt, line width=0.5mm, draw=black] at (-12, -20) [] {${2^{b-2}}$};
\svertex(10) [-{Latex[length=2.5mm]}, minimum size = 30pt, line width=0.5mm, draw=black] at (0, -20) [] {${2^{b-2}+1}$};
\svertex(11) [-{Latex[length=2.5mm]}, minimum size = 30pt, line width=0.5mm, draw=white] at (12, -20) [] {$\hdots$};
\svertex(12) [-{Latex[length=2.5mm]}, minimum size = 30pt, line width=0.5mm, draw=black] at (24, -20) [] {${2^{b-1}}$};
\svertex(13) [-{Latex[length=2.5mm]}, minimum size = 30pt, line width=0.5mm, draw=black] at (36, -20) [] {${2^{b-1}+1}$};
\svertex(14) [-{Latex[length=2.5mm]}, minimum size = 30pt, line width=0.5mm, draw=white] at (48, -20) [] {$\hdots$};
\svertex(15) [-{Latex[length=2.5mm]}, minimum size = 30pt, line width=0.5mm, draw=black] at (60, -20) [] {${3 \cdot 2^{b-2}}$};
\svertex(16) [-{Latex[length=2.5mm]}, minimum size = 30pt, line width=0.5mm, draw=black] at (72, -20) [] {${3 \cdot 2^{b-2}+1}$};
\svertex(17) [-{Latex[length=2.5mm]}, minimum size = 30pt, line width=0.5mm, draw=white] at (84, -20) [] {$\hdots$};
\svertex(18) [-{Latex[length=2.5mm]}, minimum size = 30pt, line width=0.5mm, draw=black] at (96, -20) [] {${2^{b}}$};
\draw [-] [line width=0.5mm] (7) to node {} (8);
\draw [-] [line width=0.5mm] (8) to node {} (9);
\draw [-] [line width=0.5mm, dashed] (9) to node {} (10);
\draw [-] [line width=0.5mm] (10) to node {} (11);
\draw [-] [line width=0.5mm] (11) to node {} (12);
\draw[rotate=0,red,thick,dashed] (-24, -20) ellipse (17 and 9) (-24, -14) node{$A_{\pre}^{1,0}$};
\draw[rotate=0, blue,thick,dashed] (12, -20) ellipse (17 and 9) (12, -14) node{$B_{\pre}^{1,0}$};

\draw [-] [line width=0.5mm] (13) to node {} (14);
\draw [-] [line width=0.5mm] (14) to node {} (15);
\draw [-] [line width=0.5mm, dashed] (15) to node {} (16);
\draw [-] [line width=0.5mm] (16) to node {} (17);
\draw [-] [line width=0.5mm] (17) to node {} (18);
\draw[rotate=0,red,thick,dashed] (48, -20) ellipse (17 and 9) (48, -14) node{$A_{\pre}^{1,1}$};
\draw[rotate=0, blue,thick,dashed] (84, -20) ellipse (18 and 9) (84, -14) node{$B_{\pre}^{1,1}$};

\svertex(17) [-{Latex[length=2.5mm]}, minimum size = 30pt, line width=0.5mm, draw=white] at (30, -30) [] {\large $\vdots$};

\end{tikzpicture}
}
\caption{A demonstration of Algorithm~\ref{alg:biclique_sep} given a junction tree of $\CDC(\mathcal{S}) = \SOS k(N)$, i.e. $\mathcal{S} := \{S^i\}_{i=1}^d = \{\llbracket i, i+k-1\rrbracket \}_{i=1}^{2^b}$. Note that block $i$ represents $S^i$.}
\label{fig:sosk_visual}
\end{figure}


\begin{proposition} \label{prop:sosk_ub2}
Assume that $N = 2^b + k - 1$, $k \geq 2$, and $b$ are positive integers. Let $J = \llbracket N\rrbracket $ and take $\mathcal{S}$ to correspond to the $\SOS k(N)$ constraint on $J$, i.e. $\mathcal{S} := \{\llbracket i, i+k-1\rrbracket \}_{i=1}^{2^b}$. Take $\{\{A^{i, j}, B^{i, j}\}: i \in \{0, \hdots, b-1\}, j \in \{0, \hdots, 2^i-1\}\}$ as defined in Proposition~\ref{prop:sosk_ub} and $\alpha^i = \lceil \frac{k-1+2^{b-i-1}}{2^{b-i}} \rceil$ for $i \in \{0, \hdots, b-1\}$. Then, $\{L^{i, p}, R^{i, p}\}$ is a biclique subgraph of the conflict graph $G^c_{\mathcal{S}}$ for any integers $i \in \{0, \hdots, b-1\}$ and $0 \leq p \leq \min\{\alpha^i-1, 2^i-1\}$:

\begin{subequations}
\begin{align}
    L^{i, p} &= \left(\bigcup_{q = 0}^{\lfloor \frac{2^i-1 - p}{2 \alpha^i} \rfloor} A^{i,2q \alpha^i + p}\right) \cup \left(\bigcup_{q = 0}^{\lfloor \frac{2^i-1 - p}{2 \alpha^i} - \frac{1}{2} \rfloor} B^{i, (2q+1) \alpha^i + p} \right) \\
    R^{i, p} &= \left(\bigcup_{q = 0}^{\lfloor \frac{2^i-1 - p}{2 \alpha^i} \rfloor} B^{i,2q \alpha^i + p}\right) \cup \left(\bigcup_{q = 0}^{\lfloor \frac{2^i-1 - p}{2 \alpha^i} - \frac{1}{2} \rfloor} A^{i, (2q+1) \alpha^i + p} \right).
\end{align}
\end{subequations}

\end{proposition}

\begin{proof} 
It is sufficient for us to show that $\{L, R\}$ is a biclique subgraph of $\mathcal{G}^c_{\mathcal{S}}$, where
\begin{align*}
    L &= \left(\bigcup_{q = 0}^{\lfloor \frac{2^i-1}{2 \alpha^i} \rfloor} A^{i,2q \alpha^i}\right) \cup \left(\bigcup_{q = 0}^{\lfloor \frac{2^i-1}{2 \alpha^i} - \frac{1}{2} \rfloor} B^{i, (2q+1) \alpha^i} \right) \\
    R &= \left(\bigcup_{q = 0}^{\lfloor \frac{2^i-1}{2 \alpha^i} \rfloor} B^{i,2q \alpha^i}\right) \cup \left(\bigcup_{q = 0}^{\lfloor \frac{2^i-1}{2 \alpha^i} - \frac{1}{2} \rfloor} A^{i, (2q+1) \alpha^i} \right)\\
\end{align*}

\noindent or equivalently
\begin{align*}
    L &= A^{i,0} \cup B^{i, \alpha^i} \cup A^{i, 2\alpha^i} \cup \hdots \\
    R &= B^{i,0} \cup A^{i, \alpha^i} \cup B^{i, 2\alpha^i} \cup \hdots
\end{align*}

In order to check whether $\{L, R\}$ is a biclique subgraph of the conflict graph of $\SOS k(N)$, it is sufficient to show that $\dist(L, R) \geq k$. A fact is that all the indices in $A^{i,j}$ are smaller than $B^{i, j}$. Also, the indices in $B^{i, j}$ are smaller than $A^{i, j+1}$. Also note that $\dist(A^{i,j}, B^{i,j}) \geq k$. Hence, we only need to prove $\dist(A^{i, (l-1)\alpha^i}, A^{i, l\alpha^i}) \geq k$ and $\dist(B^{i, (l-1)\alpha^i}, B^{i, l\alpha^i}) \geq k$.


We start with $l=1$. The indices in $A^{i, 0}$ are $\llbracket 1, 2^{b-i-1}\rrbracket$ and indices in $B^{i, 0}$ are $\llbracket 2^{b-i-1} + k, 2^{b-i} + k - 1\rrbracket$. The indices in $A^{i, \alpha^i}, B^{i, \alpha^i}$ are $\llbracket 1 + \alpha^i 2^{b-i}, (2\alpha^i +1) 2^{b-i-1}\rrbracket$ and $\llbracket (2\alpha^i +1) 2^{b-i-1} + k, (\alpha^i +1)2^{b-i} + k-1\rrbracket$, respectively. Since
\begin{align*}
    \alpha 2^{b-i} = \left\lceil \frac{k-1+2^{b-i-1}}{2^{b-i}} \right\rceil 2^{b-i} \geq k-1 + 2^{b-i-1},
\end{align*}

\noindent it is not hard to check that $\dist(A^{i, 0}, A^{i,\alpha^i}) \geq k$ and $\dist(B^{i, 0}, B^{i, \alpha^i}) \geq k$. By shifting the indices, we can get $\dist(A^{i, (l-1)\alpha^i}, A^{i, l\alpha^i}) \geq k$ and \\ $\dist(B^{i, (l-1)\alpha^i}, B^{i, l\alpha^i}) \geq k$ for all positive integer $l$.
\qed\end{proof}

\begin{lemma} \label{lm:bc_size}
Given two positive integers: $b \geq 1$ and $k \geq 2$,
\begin{align}
    \sum_{i=0}^{b-1} \min\left\{2^i, \left\lceil \frac{k-1+2^{b-i-1}}{2^{b-i}} \right\rceil\right\} \leq b + k - 2.
\end{align}
\end{lemma}

See Appendix~\ref{ap:lm_proof} for the proof of Lemma~\ref{lm:bc_size}. After merging bicliques as described in Proposition~\ref{prop:sosk_ub2}, the biclique cover for the conflict graph of $\SOS k(N)$ has a size of no more than $b+k-2$, if we have $N = 2^b + k - 1$ and two integers $k \geq 2$, $b \geq 1$.

\begin{theorem} \label{thm:sosk_ub}
Assume that $N = 2^b + k - 1$, $k \geq 2$, and $b$ are positive integers. Let $J = \llbracket N\rrbracket $ and take $\mathcal{S}$ to correspond to the $\SOS k(N)$ constraint on $J$, i.e. $\mathcal{S} := \{\llbracket i, i+k-1\rrbracket \}_{i=1}^{2^b}$. As defined in Proposition~\ref{prop:sosk_ub2}, take $\alpha^i = \lceil \frac{k-1+2^{b-i-1}}{2^{b-i}} \rceil$. Then, $\{\{L^{i,p}, R^{i, p}\}: i \in \{0, \hdots, b-1\}, p \in \{0, \hdots, \min\{\alpha^i-1, 2^i-1\}\}\}$ is a biclique cover of the conflict graph $G^c_{\mathcal{S}}$ with size at most $b + k - 2$.
\end{theorem}

\begin{proof}

By Propositions~\ref{prop:sosk_ub} and \ref{prop:sosk_ub2}, we know that $\{\{L^{i,p}, R^{i, p}\}: i \in \{0, \hdots, b-1\}, p \in \{0, \hdots, \min\{\alpha^i-1, 2^i-1\}\}\}$ is a biclique cover of the conflict graph $G^c_{\mathcal{S}}$. It is clear that the size of the biclique cover is at most $b+k-2$ by Lemma~\ref{lm:bc_size}.

\qed\end{proof}

In Theorem~\ref{thm:sosk_ub}, we provide an approach to build a biclique cover of the conlict graph of $\SOS k(N)$ if we have $N = 2^b+k-1$ and two integers $k \geq 2$, $b \geq 1$. In Theorem~\ref{thm:sosk_ub_all}, we show that it is easy to extend the result to $N > k \geq 2$ so that we are able to build small, ideal, and non-extended MIP formulations of $\SOS k(N)$.

\begin{theorem} \label{thm:sosk_ub_all}
Suppose that two integers $N > k \geq 2$ and $b = \lceil \log_2(N - k + 1) \rceil$. Let $\{\{L^{i,p}, R^{i, p}\}: i \in \{0, \hdots, b-1\}, p \in \{0, \hdots, \min\{\alpha^i-1, 2^i-1\}\}\}$ as defined in Proposition~\ref{prop:sosk_ub2}. The following is an ideal MIP formulation of $\SOS k(N)$:
\begin{subequations} \label{eq:sosk_mip}
\begin{alignat}{2}
& \sum_{v \in L^{i,p} \cap \llbracket N \rrbracket} \lambda_v \leq z_{i,p}, & \forall i \in \llbracket 0, b-1 \rrbracket, p \in \llbracket 0, \min\{\alpha^i-1, 2^i-1\} \rrbracket\\
& \sum_{v \in R^{i,p} \cap \llbracket N \rrbracket} \lambda_v \leq 1 - z_{i, p}, \quad & \forall i \in \llbracket 0, b-1 \rrbracket, p \in \llbracket 0, \min\{\alpha^i-1, 2^i-1\} \rrbracket\\
& z_{i, p} \in \{0, 1\}, \qquad & \forall i \in \llbracket 0, b-1 \rrbracket, p \in \llbracket 0, \min\{\alpha^i-1, 2^i-1\} \rrbracket,
\end{alignat}
\end{subequations}

\noindent where the number of binary variables in $z$ is at most $\lceil \log_2(N - k + 1) \rceil + k - 2$ and the number of constraints in~(\ref{eq:sosk_mip}) is at most $2\lceil \log_2(N - k + 1) \rceil + 2k - 4$.
\end{theorem}

\begin{proof}
By Theorem~\ref{thm:sosk_ub}, $\{\{L^{i,p}, R^{i, p}\}: i \in \{0, \hdots, b-1\}, p \in \{0, \hdots, \min\{\alpha^i-1, 2^i-1\}\}\}$ is a biclique cover of the conflict graph of $\SOS k(2^b+k-1)$ with a size at most $b + k - 2$. By Proposition~\ref{prop:larger_SOS}, $\{\{L^{i,p} \cap \llbracket N \rrbracket, R^{i, p} \cap \llbracket N \rrbracket\}: i \in \{0, \hdots, b-1\}, p \in \{0, \hdots, \min\{\alpha^i-1, 2^i-1\}\}\}$ is a biclique cover of the conflict graph of $\SOS k(N)$. By Corollary~\ref{cor:cdc_bc}, \eqref{eq:sosk_mip} is an ideal MIP formulation of $\SOS k(N)$. 
\qed\end{proof}

We compare \eqref{eq:sosk_mip} with the MIP formulation obtained by Huchette and Vielma (Theorem 5 of~\cite{huchette2019combinatorial}) and we can claim that \eqref{eq:sosk_mip} uses fewer auxiliary binary variables for any $N > k \geq 2$. We also want to note that the number of auxiliary binary variables in~\eqref{eq:sosk_mip} is at most $\frac{C + 1}{3C}$ of that in the formulation of Huchette and Vielma~\cite{huchette2019combinatorial} in he regime where $k > C \lceil\log_2(N) \rceil$ for some $C > 0$.

\begin{proposition}\label{prop:sosk_less}
Given two integers $N > k \geq 2$, $\lceil \log_2(N - k + 1) \rceil + k - 2 < \lceil \log_2(\lceil N / k \rceil - 1) \rceil + 3k$.
\end{proposition} 
\begin{proposition} \label{prop:sosk_proportion}
Given two integers $N > k \geq 2$ and $k > C \lceil \log_2(N) \rceil$ for some $C > \frac{1}{2}$, $$\frac{\lceil \log_2(N - k + 1) \rceil + k - 2}{\lceil \log_2(\lceil N / k \rceil - 1) \rceil + 3k} < \frac{C + 1}{3 C} < 1.$$
\end{proposition}

See Appendix~\ref{ap:proof_sosk_number} for the proofs of Propositions~\ref{prop:sosk_less} and~\ref{prop:sosk_proportion}.

\section{New Ideal Formulations of Combinatorial Disjunctive Constraints via Junction Trees}





We now introduce a procedure to represent any combinatorial disjunctive constraint $\lambda \in \CDC(\mathcal{S})$ by a new constraint $\lambda' \in \CDC(\mathcal{S}')$ admitting a junction tree at a cost of potentially using more continuous variables. Algorithm~\ref{alg:build_CDC} demonstrates how to obtain $\mathcal{S}'$ from $\mathcal{S}$ and a mapping function $\alpha$ to relate the indices in $\lambda'$ and $\lambda$. Proposition~\ref{prop:reformulate_cdc_tree} shows that $\CDC(\mathcal{S}')$ returned by Algorithm~\ref{alg:build_CDC} admits a junction tree and how to use $\CDC(\mathcal{S}')$ to formulate $\CDC(\mathcal{S})$. We eventually end up with two ideal formulations which use less variables than Jeroslow and Lowe's formulation in~\eqref{form:cdc_jl}.

\begin{algorithm}[h]
\begin{algorithmic}[1]
\State \textbf{Input}: A combinatorial disjunctive constraint $\CDC(\mathcal{S})$ where $\mathcal{S} := \{S^i\}_{i=1}^d$; A boolean variable $\disjoint \in \{\True, \False\}$. \Comment{If $\disjoint = \True$, we construct a $\mathcal{S}'$ with pairwise disjoint index sets to represent $\CDC(\mathcal{S})$. Otherwise, we a $\CDC(\mathcal{S}')$ admitting a junction tree but not necessary with pairwise disjoint index sets in $\mathcal{S}'$.}
\State \textbf{Output}: A collection of index sets $\mathcal{S}'$ and a mapping function $\alpha: \cup_{S' \in \mathcal{S}'} S' \rightarrow \cup_{S \in \mathcal{S}} S$.
\State $c := 0$. \label{ln:build_CDC_start}
\For{$i \in \llbracket d \rrbracket$}
\State $S'^{,i} := \{c+1, \hdots, c+|S^i|\}$. 
\State $c := c + |S^i|$.
\State $\alpha(S'^{,i}_l) := S^i_l$ for all $l \in \llbracket |S^i| \rrbracket$.
\EndFor \label{ln:build_CDC_end}
\If{$\disjoint$}
\State Let $\mathcal{T}'$ be a tree where the vertices are $\mathcal{S}' := \{S'^{,i}\}_{i=1}^d$ and edges are $\{S'^{,i}S'^{,i+1}: i \in \llbracket d-1 \rrbracket\}$.
\State \textbf{return} $\mathcal{S}', \alpha$, $\mathcal{T}'$
\EndIf
\State Construct an intersection graph $\mathcal{K}_{\mathcal{S}}$ of the disjunctive constraint $\CDC(\mathcal{S})$.
\State Calculate a maximum spanning tree of $\mathcal{K}_{\mathcal{S}}$, $\mathcal{T}$. \label{ln:build_CDC_tree}
\State $\mathcal{S}_{\select} := \{S^1\}$.
\While{$|\mathcal{S}_{\select}| < d$}
\State Find an edge $e \in \mathcal{T}$ such that its one end $S^i \in \mathcal{S}_{\select}$ and the other end $S^j \in \mathcal{S} \setminus \mathcal{S}_{\select}$.
\State $\mathcal{S}_{\select} := \mathcal{S}_{\select} \cup \{S^j\}$.
\State For every $x \in S'^{,i}$ and $y \in S'^{,j}$, if $\alpha(x) = \alpha(y)$, then $S'^{,j} := S'^{,j} \cup \{x\} \setminus \{y\}$.
\EndWhile
\State Let $\mathcal{S}' := \{S'^{,i}\}_{i=1}^d$. Remove $\alpha(x)$ for all $x \not\in \cup_{S' \in \mathcal{S}'} S'$.
\State Let $\mathcal{T}'$ be a tree by replacing all the vertices $S^i$ to $S'^{,i}$ in the tree $\mathcal{T}$.
\State \textbf{return} $\mathcal{S}', \alpha$, $\mathcal{T}'$
\end{algorithmic}
\caption{Find a $\CDC(\mathcal{S}')$ admitting a junction tree to represent an arbitrary combinatorial disjunctive constraint $\CDC(\mathcal{S})$.} \label{alg:build_CDC}
\end{algorithm}

\begin{proposition} \label{prop:reformulate_cdc_tree}
Given a combinatorial disjunctive constraint $\lambda \in \CDC(\mathcal{S})$ where $\mathcal{S} := \{S^i\}_{i=1}^d$, let $\mathcal{S}'$, $\alpha$, and $\mathcal{T}'$ be the output of Algorithm~\ref{alg:build_CDC} with $\CDC(\mathcal{S})$ and $\disjoint \in \{\False, \True\}$ as inputs. Then, $\mathcal{T}'$ is a junction tree of $\CDC(\mathcal{S}')$ and
\begin{align} \label{eq:cdc_proj}
    \CDC(\mathcal{S}) = \Proj_{\lambda} \left\{(\lambda, \lambda'): \lambda_v = \sum_{u: \alpha(u) = v} \lambda'_u, \quad \lambda' \in \CDC(\mathcal{S}') \right\}.
\end{align}
\end{proposition}

\begin{proof}

In the case that $\disjoint := \False$, it is easy to see that $\mathcal{T}'$ is a junction tree of $\CDC(\mathcal{S}')$ where the vertices are $\mathcal{S}' := \{S'^{,i}\}_{i=1}^d$ and edges are $\{S'^{,i}S'^{,i+1}: i \in \llbracket d-1 \rrbracket\}$. It is also not hard to see that~\eqref{eq:cdc_proj} holds in this case.

In the other case, by replacing all the vertices $S^i$ to $S'^{,i}$ in the tree $\mathcal{T}$ to get $\mathcal{T}'$ (also change the middle sets and weights on edges), then it is not hard to see that each edge $e$ in $\mathcal{T}'$ separate vertices into $\mathcal{S}'_1$ and $\mathcal{S}'_2$ such that $\bigcup_{S \in \mathcal{S}'_1} S \cap \bigcup_{S \in \mathcal{S}'_2} S \subseteq \MID(e)$. By Proposition~\ref{prop:check_junction_tree}, we know that $\mathcal{T}'$ is a junction tree of $\CDC(\mathcal{S})$.

Given an arbitrary $i \in \llbracket d \rrbracket$, if $\lambda \in Q(S^i)$, then we let $\lambda' \in Q(S'^{,i})$ such that $\lambda'_u = \lambda_v$ for $u \in S'^{,i}$ and $\alpha(u) = v$. Since there exists an one-to-one corresponding between the elements in $S^i$ and $S'^{,i}$, then we can prove ``$\subseteq$" in \eqref{eq:cdc_proj}. The proof of ``$\supseteq$" in \eqref{eq:cdc_proj} follows the same manner. Suppose that $\lambda' \in Q(S'^{,i})$, then we can prove that $\lambda \in Q(S^i)$ where $\lambda_v = \sum_{u: \alpha(u) = v} \lambda'_u$.
\qed\end{proof}

Then, we can write down a formulation of $\CDC(\mathcal{S})$ by using $\CDC(\mathcal{S}')$.

\begin{remark} \label{rm:tree1}
Let $\mathcal{S}'$, $\alpha$, and $\mathcal{T}'$ be the output of Algorithm~\ref{alg:build_CDC} with $\CDC(\mathcal{S})$ and $\disjoint := \False$ as inputs. Let $\{\{A'^{,i}, B'^{,i}\}\}_{i=1}^t$ be the biclique cover of the conflict graph of $\CDC(\mathcal{S}')$ returned by Algorithm~\ref{alg:biclique}. Denote that $J' = \bigcup_{S' \in \mathcal{S}'} S'$. Then an ideal formulation of $\CDC(\mathcal{S})$ can be written as
\begin{subequations} \label{form:cdc_tree}
\begin{alignat}{2}
    &\lambda_v = \sum_{u: \alpha(u) = v} \lambda'_u, & \forall v \in J \label{form:cdc_tree_a}\\
    & \sum_{v \in A'^{,j}} \lambda_v \leq z_j, & \forall j \in \llbracket t\rrbracket\\
    & \sum_{v \in B'^{,j}} \lambda_v \leq 1 - z_j, \quad & \forall j \in \llbracket t\rrbracket\\
    & \lambda' \in \Delta^{J'} \\
    & z_j \in \{0, 1\}, & \forall j \in \llbracket t\rrbracket.
\end{alignat}
\end{subequations}

If we project $\lambda$ out using \eqref{form:cdc_tree_a}, the number of auxiliary continuous variables is $\sum_{S \in \mathcal{S}} |S| - w(\mathcal{T}') - |\bigcup_{S \in \mathcal{S}} S|$. The number of binary variables is $t$, which is no more than $|\mathcal{S}'|-1 = |\mathcal{S}|-1$. The number of constraints is $2t$ besides the variable bounds. Note that $t$ is potentially $O(\log_2(|\mathcal{S}|))$, although we can only provide a theoretical bound of $|\mathcal{S}|-1$ as in Theorem~\ref{thm:alg_sep_bound_time}.
\end{remark}

Additionally, we can write down a logarithmic ideal MIP formulation of $\CDC(\mathcal{S})$ by introducing enough continuous variables to $\mathcal{S}$ so that each pair of index sets is disjoint (Algorithm~\ref{alg:build_CDC} with $\disjoint = \True$). In this way, although we need to introduce $\sum_{S \in \mathcal{S}} |S| - |\bigcup_{S \in \mathcal{S}} S|$ auxiliary continuous variables, the formulation only contains $\lceil \log_2(|\mathcal{S}|) \rceil$ binary variables.


\begin{remark} \label{rm:tree2}
Let $\mathcal{S}''$,  $\alpha''$, and $\mathcal{T}''$ be the output of Algorithm~\ref{alg:build_CDC} with $\CDC(\mathcal{S})$ and $\disjoint := \True$ as inputs. Also, let $\{\{A''^{,i}, B''^{,i}\}\}_{i=1}^{t''}$ be the biclique cover of the conflict graph of $\CDC(\mathcal{S}'')$ returned by Algorithm~\ref{alg:biclique}. Denote that $J'' = \bigcup_{S'' \in \mathcal{S}''} S''$. Then an ideal formulation of $\CDC(\mathcal{S})$ can be written as
\begin{subequations} \label{form:cdc_tree_2}
\begin{alignat}{2}
    &\lambda_v = \sum_{u: \alpha''(u) = v} \lambda''_u, & \forall v \in J \label{form:cdc_tree_2_a}\\
    & \sum_{v \in A''^{,j}} \lambda_v \leq z_j, & \forall j \in \llbracket t''\rrbracket\\
    & \sum_{v \in B''^{,j}} \lambda_v \leq 1 - z_j, \quad & \forall j \in \llbracket t''\rrbracket \\
    & \lambda'' \in \Delta^{J''} \\
    & z_j \in \{0, 1\}, & \forall j \in \llbracket t''\rrbracket.
\end{alignat}
\end{subequations}

This formulation has $\sum_{S \in \mathcal{S}} |S| - |\bigcup_{S \in \mathcal{S}} S|$ continuous variables if we project $\lambda$ out using \eqref{form:cdc_tree_2_a}. It introduces $w(\mathcal{T}')$ more continuous variables than \eqref{form:cdc_tree}, where $\mathcal{T}'$ is defined in Remark~\ref{rm:tree1}. However, the number of binary variables in~\eqref{form:cdc_tree_2} is only $t'' = \lceil \log_2(|\mathcal{S}|)\rceil$.\footnote{Since all the sets in $\mathcal{S}''$ are disjointed with each other, the bicliques in the same ``level" in the recursion of Algorithm~\ref{alg:biclique_sep} can be merged into one and there are only $\lceil \log_2(|\mathcal{S}|)\rceil$ levels.} It also has the same number of auxiliary continuous and binary variables as \eqref{form:cdc_vielma}. However, \eqref{form:cdc_tree_2} has $2\lceil \log_2(|\mathcal{S}|)\rceil$ inequality constraints, whereas \eqref{form:cdc_vielma} uses $\lceil \log_2(|\mathcal{S}|)\rceil$ equality constraints.

\end{remark}



\section{More Applications}

In this section, we discuss two more applications of CDCs admitting junction trees: univariate piecewise linear functions and planar obstacle avoidance.

\subsection{Univariate Piecewise Linear Functions}

Suppose that we are interested in formulating a univariate piecewise linear function $y = f(x)$, where $f$ has $N$ breakpoints: $x_1 < x_2 < \hdots < x_N$. Then, $\{(x, y): y=f(x), x \in [x_1, x_N]\}$ can be modeled by a special ordered set type 2, $\SOS 2(N)$:
\begin{subequations} \label{eq:pwl_sos2}
\begin{alignat}{2}
    & y = \sum_{v=1}^N \lambda_v f(x_v), \qquad & x = \sum_{v=1}^N \lambda_v x_v\\
    & \lambda \in \SOS 2(N), & x, y \in \mathbb{R}.
\end{alignat}
\end{subequations}

Clearly, $\SOS 2(N)$ is a CDC admitting a junction tree. A MIP formulation for $\SOS 2(N)$ with logarithmic binary variables and constraints can be derived from~\eqref{eq:sosk_mip} in Theorem~\ref{thm:sosk_ub_all} with $k=2$. We note that this formulation~\eqref{eq:sosk_mip} for $\SOS 2(N)$ is one of the MIP formulations in~\cite[Theorem 4]{vielma2011modeling} and is equivalent to the \textit{logarithmic independent
branching} (LogIB) formulation in~\cite{huchette2017nonconvex}, which can be derived from~\cite[Proposition 5]{huchette2019combinatorial} with the binary reflected Gray code as an input.

\subsection{Planar Obstacle Avoidance}

Consider a routing problem of a robot or unmanned aerial vehicle (UAV) where the robot or vehicle must avoid some fixed number of obstacles in a bounded planar area. To formulate the obstacle avoidance constraint, the location of the robot or vehicle $x^t \in \mathbb{R}^2$ must stay in the obstacle-free area at any given time $t$: $x^t \in \Omega$, where $\Omega$ describes a bounded (nonconvex) obstacle-free region. Modeling this constraint into a MIP formulation of path planning problem draws attention from both the operations research~\cite{huchette2019combinatorial} and robotics~\cite{mellinger2012mixed,deits2014footstep,deits2015efficient,prodan2016mixed} communities.The obstacle-free region $\Omega$ can be represented by a union of a finite and fixed number of polyhedra $P^i$ for $i \in \llbracket d \rrbracket$ as shown in~\eqref{eq:dc_poly}. We note that there are many ways to partition $\Omega$, i.e., the inner areas of $P^i$ and $P^j$ are disjoint: $\relint(P^i) \cap \relint(P^j) = \emptyset$ for $i \neq j$.


\begin{figure}[H]
    \centering
\begin{minipage}{0.5\textwidth}
\tikzset{every picture/.style={line width=0.75pt}} 
\begin{tikzpicture}[x=0.75pt,y=0.75pt,yscale=-1,xscale=1, scale=0.75]

\draw  [fill={rgb, 255:red, 0; green, 0; blue, 0 }  ,fill opacity=0.3 ] (107.61,35.91) -- (219.26,40.39) -- (167.24,69.82) -- (61.51,65.01) -- cycle ;
\draw  [fill={rgb, 255:red, 0; green, 0; blue, 0 }  ,fill opacity=0.3 ] (219.26,40.39) -- (276.62,131.28) -- (219.71,123.26) -- (167.24,69.82) -- cycle ;
\draw  [fill={rgb, 255:red, 0; green, 0; blue, 0 }  ,fill opacity=0.3 ] (276.62,131.28) -- (222.06,221.96) -- (181.21,193.57) -- (219.71,123.26) -- cycle ;
\draw  [fill={rgb, 255:red, 0; green, 0; blue, 0 }  ,fill opacity=0.3 ] (181.21,193.57) -- (222.06,221.96) -- (112.17,242.27) -- (113.84,198.82) -- (142.03,181.63) -- cycle ;
\draw  [fill={rgb, 255:red, 0; green, 0; blue, 0 }  ,fill opacity=0.3 ] (36.49,201.75) -- (58.21,176.12) -- (113.84,198.82) -- (112.17,242.27) -- (49.62,230.67) -- cycle ;
\draw  [fill={rgb, 255:red, 0; green, 0; blue, 0 }  ,fill opacity=0.3 ] (159.62,142.7) -- (181.21,193.57) -- (142.03,181.63) -- (131.96,151.9) -- cycle ;
\draw    (6.82,134.66) -- (36.49,201.75) ;
\draw    (61.51,65.01) -- (6.82,134.66) ;
\draw  [fill={rgb, 255:red, 0; green, 0; blue, 0 }  ,fill opacity=0.3 ] (62.25,108.85) -- (60.98,140.84) -- (21.65,168.2) -- (6.82,134.66) -- (24.26,115.02) -- cycle ;
\draw  [fill={rgb, 255:red, 0; green, 0; blue, 0 }  ,fill opacity=0.3 ] (124.24,110.36) -- (96.97,142.13) -- (60.98,140.84) -- (62.25,108.85) -- cycle ;
\draw  [fill={rgb, 255:red, 0; green, 0; blue, 0 }  ,fill opacity=0.3 ] (124.24,110.36) -- (159.62,142.7) -- (131.96,151.9) -- (96.97,142.13) -- cycle ;

\draw (25,123) node [anchor=north west][inner sep=0.75pt]   [align=left] {$\displaystyle S^{1}$};
\draw (73,115) node [anchor=north west][inner sep=0.75pt]   [align=left] {$\displaystyle S^{2}$};
\draw (119,124) node [anchor=north west][inner sep=0.75pt]   [align=left] {$\displaystyle S^{3}$};
\draw (145,156) node [anchor=north west][inner sep=0.75pt]   [align=left] {$\displaystyle S^{4}$};
\draw (145,201) node [anchor=north west][inner sep=0.75pt]   [align=left] {$\displaystyle S^{5}$};
\draw (67,199) node [anchor=north west][inner sep=0.75pt]   [align=left] {$\displaystyle S^{6}$};
\draw (122,42) node [anchor=north west][inner sep=0.75pt]   [align=left] {$\displaystyle S^{9}$};
\draw (219,157) node [anchor=north west][inner sep=0.75pt]   [align=left] {$\displaystyle S^{7}$};
\draw (211,77) node [anchor=north west][inner sep=0.75pt]   [align=left] {$\displaystyle S^{8}$};
\end{tikzpicture}
\end{minipage}
\begin{minipage}{0.49\textwidth}
\tikzset{every picture/.style={line width=0.75pt}} 
\begin{tikzpicture}[x=0.75pt,y=0.75pt,yscale=-1,xscale=1,scale=0.8]

\draw   (24.87,105.57) .. controls (24.87,96.42) and (32.28,89) .. (41.43,89) .. controls (50.58,89) and (58,96.42) .. (58,105.57) .. controls (58,114.72) and (50.58,122.13) .. (41.43,122.13) .. controls (32.28,122.13) and (24.87,114.72) .. (24.87,105.57) -- cycle ;
\draw   (77.87,97.57) .. controls (77.87,88.42) and (85.28,81) .. (94.43,81) .. controls (103.58,81) and (111,88.42) .. (111,97.57) .. controls (111,106.72) and (103.58,114.13) .. (94.43,114.13) .. controls (85.28,114.13) and (77.87,106.72) .. (77.87,97.57) -- cycle ;
\draw   (120.87,120.57) .. controls (120.87,111.42) and (128.28,104) .. (137.43,104) .. controls (146.58,104) and (154,111.42) .. (154,120.57) .. controls (154,129.72) and (146.58,137.13) .. (137.43,137.13) .. controls (128.28,137.13) and (120.87,129.72) .. (120.87,120.57) -- cycle ;
\draw   (158.87,157.57) .. controls (158.87,148.42) and (166.28,141) .. (175.43,141) .. controls (184.58,141) and (192,148.42) .. (192,157.57) .. controls (192,166.72) and (184.58,174.13) .. (175.43,174.13) .. controls (166.28,174.13) and (158.87,166.72) .. (158.87,157.57) -- cycle ;
\draw   (156.87,214.57) .. controls (156.87,205.42) and (164.28,198) .. (173.43,198) .. controls (182.58,198) and (190,205.42) .. (190,214.57) .. controls (190,223.72) and (182.58,231.13) .. (173.43,231.13) .. controls (164.28,231.13) and (156.87,223.72) .. (156.87,214.57) -- cycle ;
\draw   (73.87,195.57) .. controls (73.87,186.42) and (81.28,179) .. (90.43,179) .. controls (99.58,179) and (107,186.42) .. (107,195.57) .. controls (107,204.72) and (99.58,212.13) .. (90.43,212.13) .. controls (81.28,212.13) and (73.87,204.72) .. (73.87,195.57) -- cycle ;
\draw   (236.87,163.57) .. controls (236.87,154.42) and (244.28,147) .. (253.43,147) .. controls (262.58,147) and (270,154.42) .. (270,163.57) .. controls (270,172.72) and (262.58,180.13) .. (253.43,180.13) .. controls (244.28,180.13) and (236.87,172.72) .. (236.87,163.57) -- cycle ;
\draw   (214.87,79.57) .. controls (214.87,70.42) and (222.28,63) .. (231.43,63) .. controls (240.58,63) and (248,70.42) .. (248,79.57) .. controls (248,88.72) and (240.58,96.13) .. (231.43,96.13) .. controls (222.28,96.13) and (214.87,88.72) .. (214.87,79.57) -- cycle ;
\draw   (139.87,30.57) .. controls (139.87,21.42) and (147.28,14) .. (156.43,14) .. controls (165.58,14) and (173,21.42) .. (173,30.57) .. controls (173,39.72) and (165.58,47.13) .. (156.43,47.13) .. controls (147.28,47.13) and (139.87,39.72) .. (139.87,30.57) -- cycle ;
\draw    (58,105.57) -- (77.87,97.57) ;
\draw    (121.67,113.5) -- (110,105.09) ;
\draw    (162.67,147.17) -- (149.33,132.83) ;
\draw    (173.43,198) -- (173,174.5) ;
\draw    (156.87,214.57) -- (105.67,201.5) ;
\draw    (189.33,211.17) -- (241,174.83) ;
\draw    (236.33,95.17) -- (250.33,147.17) ;
\draw    (170.67,39.5) -- (217,70.5) ;

\draw (33,95) node [anchor=north west][inner sep=0.75pt]   [align=left] {$\displaystyle S^{1}$};
\draw (86,87) node [anchor=north west][inner sep=0.75pt]   [align=left] {$\displaystyle S^{2}$};
\draw (129,110) node [anchor=north west][inner sep=0.75pt]   [align=left] {$\displaystyle S^{3}$};
\draw (167,147) node [anchor=north west][inner sep=0.75pt]   [align=left] {$\displaystyle S^{4}$};
\draw (165,204) node [anchor=north west][inner sep=0.75pt]   [align=left] {$\displaystyle S^{5}$};
\draw (82,185) node [anchor=north west][inner sep=0.75pt]   [align=left] {$\displaystyle S^{6}$};
\draw (245,153) node [anchor=north west][inner sep=0.75pt]   [align=left] {$\displaystyle S^{7}$};
\draw (223,69) node [anchor=north west][inner sep=0.75pt]   [align=left] {$\displaystyle S^{8}$};
\draw (148,20) node [anchor=north west][inner sep=0.75pt]   [align=left] {$\displaystyle S^{9}$};
\end{tikzpicture}
\end{minipage}
\caption{A convex polyhedron partition of a nonconvex region (grey area) in a plane (left) and a junction tree of its corresponding combinatorial disjunctive constraint.}
\label{fig:obstacle_tree}
\end{figure}
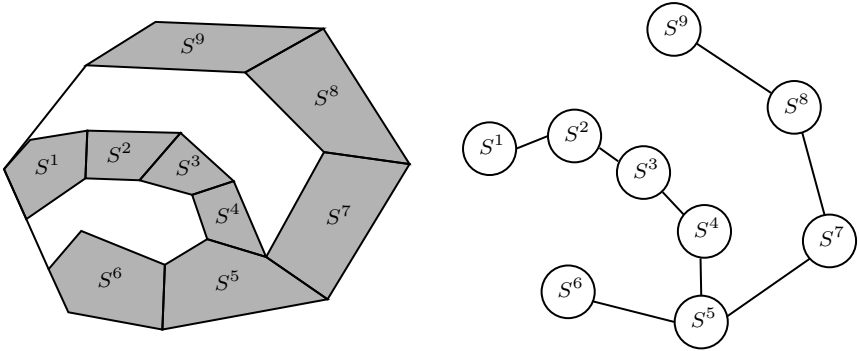

Since the region $\Omega \subset \mathbb{R}^2$ is bounded, each $P^i$ can be represented by the convex combination of its extreme points by the Minkowski-Weyl Theorem. Thus, the corresponding combinatorial disjunctive constraint is given by $\mathcal{S} := \{\ext(P^i)\}_{i=1}^d$ and $J := \cup_{S \in \mathcal{S}} S$. As shown in Figure~\ref{fig:obstacle_tree}, we can see that many corresponding CDCs of obstacle-free regions $\Omega$'s admit junction trees: when $\Omega$ can be partitioned into a ``path" or ``tree" of convex regions. In Figure~\ref{fig:obstacle_tree}, $S^4$ and $S^7$ share an extreme point but are not connected directly. It is easy to check that a path: $S^4 S^5 S^7$ implies the existence of a junction tree.

It is also worth noticing that CDCs admitting junction trees are pairwise IB-representable (Theorem~\ref{thm:IB_tree}). However, not every corresponding CDC of any partition of $\Omega$ is pairwise IB-presentable. For example, the corresponding CDC to the left of Figure~\ref{fig:obstacle_not_pairwise_IB} is not pairwise IB-representable. Since $\{v_1, v_2\} \subset S^1$, $\{v_1, v_3\} \subset S^2$, and $\{v_2, v_3\} \subset S^4$, all of $\{v_1, v_2\}, \{v_1, v_3\},$ and $\{v_2, v_3\}$ are feasible sets (Definition~\ref{def:feasible_sets}). However, $\{v_1, v_2, v_3\} \not\subseteq S^i$ for any $i \in \llbracket 8 \rrbracket$ and it is a minimal infeasible set. Thus, the corresponding CDC in Figure~\ref{fig:obstacle_not_pairwise_IB} is not pairwise IB-representable. However, if we replace $\{v_3, v_4\}$ with $\{v_{3, 2}, v_{4, 2}\}$ in $S^2$ and $\{v_3, v_5\}$ with $\{v_{3, 5}, v_{4, 5}\}$ in $S^5$, the new combinatorial disjunctive constraint admits a junction tree as shown on the right of Figure~\ref{fig:obstacle_not_pairwise_IB}. Therefore, it is pairwise IB-representable with only a cost of introducing 4 auxiliary continuous variables. Modeling the nonconvex region in Figure~\ref{fig:obstacle_not_pairwise_IB} could be an example to show that \eqref{form:cdc_tree} might be more computationally efficient than \eqref{form:cdc_vielma} and \eqref{form:cdc_tree_2} since fewer number of continuous variables are used. We will then study the number of continuous variables saved by \eqref{form:cdc_tree} comparing with \eqref{form:cdc_vielma} and \eqref{form:cdc_tree_2} in the setting of planar obstacle avoidance.

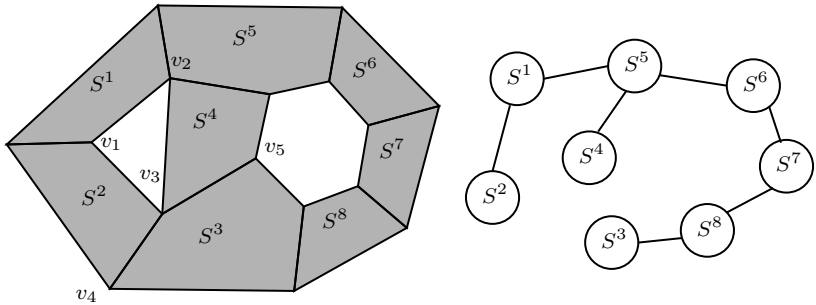
\begin{figure}[h]
    \centering
    \begin{minipage}{0.5\textwidth}
\tikzset{every picture/.style={line width=0.75pt}} 
\begin{tikzpicture}[x=0.75pt,y=0.75pt,yscale=-1,xscale=1,scale=0.9]
\draw  [fill={rgb, 255:red, 0; green, 0; blue, 0 }  ,fill opacity=0.3 ] (105.57,32.8) -- (112.35,73.44) -- (68.77,109.3) -- (21.8,110.5) -- cycle ;
\draw  [fill={rgb, 255:red, 0; green, 0; blue, 0 }  ,fill opacity=0.3 ] (108,149.35) -- (159.81,118.27) -- (186.44,145.17) -- (181.12,192.38) -- (78.94,191.19) -- cycle ;
\draw  [fill={rgb, 255:red, 0; green, 0; blue, 0 }  ,fill opacity=0.3 ] (68.77,109.3) -- (108,149.35) -- (78.94,191.19) -- (21.8,110.5) -- cycle ;
\draw  [fill={rgb, 255:red, 0; green, 0; blue, 0 }  ,fill opacity=0.3 ] (112.35,73.44) -- (167.56,82.41) -- (159.81,118.27) -- (108,149.35) -- cycle ;
\draw  [fill={rgb, 255:red, 0; green, 0; blue, 0 }  ,fill opacity=0.3 ] (112.35,73.44) -- (167.56,82.41) -- (200.49,75.24) -- (207.75,34) -- (105.57,32.8) -- cycle ;
\draw  [fill={rgb, 255:red, 0; green, 0; blue, 0 }  ,fill opacity=0.3 ] (261.5,88.98) -- (222.28,99.74) -- (200.49,75.24) -- (207.75,34) -- cycle ;
\draw  [fill={rgb, 255:red, 0; green, 0; blue, 0 }  ,fill opacity=0.3 ] (216.47,133.81) -- (243.58,157.12) -- (181.12,192.38) -- (186.44,145.17) -- cycle ;
\draw  [fill={rgb, 255:red, 0; green, 0; blue, 0 }  ,fill opacity=0.3 ] (261.5,88.98) -- (243.58,157.12) -- (216.47,133.81) -- (222.28,99.74) -- cycle ;

\draw (65.96,68.26) node [anchor=north west][inner sep=0.75pt]   [align=left] {$\displaystyle S^{1}$};
\draw (61.43,131.77) node [anchor=north west][inner sep=0.75pt]   [align=left] {$\displaystyle S^{2}$};
\draw (126.3,153.14) node [anchor=north west][inner sep=0.75pt]   [align=left] {$\displaystyle S^{3}$};
\draw (72,105) node [anchor=north west][inner sep=0.75pt]  [font=\footnotesize] [align=left] {$\displaystyle v_{1}$};
\draw (111,60) node [anchor=north west][inner sep=0.75pt]  [font=\footnotesize] [align=left] {$\displaystyle v_{2}$};
\draw (93.67,123.21) node [anchor=north west][inner sep=0.75pt]  [font=\footnotesize] [align=left] {$\displaystyle v_{3}$};
\draw (58.48,190.3) node [anchor=north west][inner sep=0.75pt]  [font=\footnotesize] [align=left] {$\displaystyle v_{4}$};
\draw (123,89.18) node [anchor=north west][inner sep=0.75pt]   [align=left] {$\displaystyle S^{4}$};
\draw (144.99,42.56) node [anchor=north west][inner sep=0.75pt]   [align=left] {$\displaystyle S^{5}$};
\draw (211.52,60.49) node [anchor=north west][inner sep=0.75pt]   [align=left] {$\displaystyle S^{6}$};
\draw (226.58,105.92) node [anchor=north west][inner sep=0.75pt]   [align=left] {$\displaystyle S^{7}$};
\draw (195.02,144.77) node [anchor=north west][inner sep=0.75pt]   [align=left] {$\displaystyle S^{8}$};
\draw (163,108) node [anchor=north west][inner sep=0.75pt]  [font=\footnotesize] [align=left] {$\displaystyle v_{5}$};
\end{tikzpicture}
    \end{minipage}
    \begin{minipage}{0.49\textwidth}
\tikzset{every picture/.style={line width=0.75pt}} 
\begin{tikzpicture}[x=0.75pt,y=0.75pt,yscale=-1,xscale=1,scale=0.8]

\draw   (66.87,73.57) .. controls (66.87,64.42) and (74.28,57) .. (83.43,57) .. controls (92.58,57) and (100,64.42) .. (100,73.57) .. controls (100,82.72) and (92.58,90.13) .. (83.43,90.13) .. controls (74.28,90.13) and (66.87,82.72) .. (66.87,73.57) -- cycle ;
\draw   (51.53,148.23) .. controls (51.53,139.08) and (58.95,131.67) .. (68.1,131.67) .. controls (77.25,131.67) and (84.67,139.08) .. (84.67,148.23) .. controls (84.67,157.38) and (77.25,164.8) .. (68.1,164.8) .. controls (58.95,164.8) and (51.53,157.38) .. (51.53,148.23) -- cycle ;
\draw   (125.87,176.57) .. controls (125.87,167.42) and (133.28,160) .. (142.43,160) .. controls (151.58,160) and (159,167.42) .. (159,176.57) .. controls (159,185.72) and (151.58,193.13) .. (142.43,193.13) .. controls (133.28,193.13) and (125.87,185.72) .. (125.87,176.57) -- cycle ;
\draw   (111.87,123.23) .. controls (111.87,114.08) and (119.28,106.67) .. (128.43,106.67) .. controls (137.58,106.67) and (145,114.08) .. (145,123.23) .. controls (145,132.38) and (137.58,139.8) .. (128.43,139.8) .. controls (119.28,139.8) and (111.87,132.38) .. (111.87,123.23) -- cycle ;
\draw   (140.2,65.57) .. controls (140.2,56.42) and (147.62,49) .. (156.77,49) .. controls (165.92,49) and (173.33,56.42) .. (173.33,65.57) .. controls (173.33,74.72) and (165.92,82.13) .. (156.77,82.13) .. controls (147.62,82.13) and (140.2,74.72) .. (140.2,65.57) -- cycle ;
\draw   (214.2,77.9) .. controls (214.2,68.75) and (221.62,61.33) .. (230.77,61.33) .. controls (239.92,61.33) and (247.33,68.75) .. (247.33,77.9) .. controls (247.33,87.05) and (239.92,94.47) .. (230.77,94.47) .. controls (221.62,94.47) and (214.2,87.05) .. (214.2,77.9) -- cycle ;
\draw   (234.87,128.57) .. controls (234.87,119.42) and (242.28,112) .. (251.43,112) .. controls (260.58,112) and (268,119.42) .. (268,128.57) .. controls (268,137.72) and (260.58,145.13) .. (251.43,145.13) .. controls (242.28,145.13) and (234.87,137.72) .. (234.87,128.57) -- cycle ;
\draw   (185.53,168.9) .. controls (185.53,159.75) and (192.95,152.33) .. (202.1,152.33) .. controls (211.25,152.33) and (218.67,159.75) .. (218.67,168.9) .. controls (218.67,178.05) and (211.25,185.47) .. (202.1,185.47) .. controls (192.95,185.47) and (185.53,178.05) .. (185.53,168.9) -- cycle ;
\draw    (100,73.57) -- (140.2,65.57) ;
\draw    (79,90.25) -- (68.1,131.67) ;
\draw    (134,106.75) -- (151.5,81.25) ;
\draw    (172,71.25) -- (214.2,77.9) ;
\draw    (248,113.75) -- (240.5,90.75) ;
\draw    (214,157.75) -- (243,142.25) ;
\draw    (159,176.57) -- (186,173.75) ;

\draw (75,63) node [anchor=north west][inner sep=0.75pt]   [align=left] {$\displaystyle S^{1}$};
\draw (59.67,137.67) node [anchor=north west][inner sep=0.75pt]   [align=left] {$\displaystyle S^{2}$};
\draw (134,166) node [anchor=north west][inner sep=0.75pt]   [align=left] {$\displaystyle S^{3}$};
\draw (120,112.67) node [anchor=north west][inner sep=0.75pt]   [align=left] {$\displaystyle S^{4}$};
\draw (148.33,55) node [anchor=north west][inner sep=0.75pt]   [align=left] {$\displaystyle S^{5}$};
\draw (222.33,67.33) node [anchor=north west][inner sep=0.75pt]   [align=left] {$\displaystyle S^{6}$};
\draw (243,118) node [anchor=north west][inner sep=0.75pt]   [align=left] {$\displaystyle S^{7}$};
\draw (193.67,158.33) node [anchor=north west][inner sep=0.75pt]   [align=left] {$\displaystyle S^{8}$};
\end{tikzpicture}
    \end{minipage}
    \caption{A convex polyhedron partition of a nonconvex region (grey area) in a plane (left) where its corresponding combinatorial disjunctive constraint is not pairwise IB-representable; A junction tree of modified version of CDC of the left where $\{v_3, v_4\}$ are replaced by $\{v_{3, 2}, v_{4, 2}\}$ in $S^2$ and $\{v_3, v_5\}$ are replaced by $\{v_{3, 5}, v_{4, 5}\}$ in $S^5$.}
    \label{fig:obstacle_not_pairwise_IB}
\end{figure}

We define the connectivity of polytope partitions and show that, given a connected polytope partitions of a bounded planar region $\Omega = \{P^i\}_{i=1}^d \subseteq \mathbb{R}^2$, the formulation~\eqref{form:cdc_tree} uses $2(d - 1)$ fewer continuous variables than \eqref{form:cdc_vielma} or \eqref{form:cdc_tree_2}.

\begin{definition}[Connectivity of Polytope Partitions]
Take bounded $\Omega \subseteq \mathbb{R}^2$ and a polytope partition $\{P^i\}_{i=1}^d$ of $\Omega$, i.e. $\relint(P^i) \cap \relint(P^j) = \emptyset$ and $\Omega = \bigcup_{i=1}^d P^i$. We can construct a dual graph $\mathcal{D}$ of the polytope partition with a vertex set of $\{P^i\}_{i=1}^d$ and $P^i$ is adjacent to $P^j$ if and only if $|P^i \cap P^j| = \infty$. We say the polytope partition $\{P^i\}_{i=1}^d$ is connected if $\mathcal{D}$ is connected.
\end{definition}

\begin{theorem} \label{thm:poly_part}
Given bounded $\Omega \subseteq \mathbb{R}^2$ and a connected polytope partition $\{P^i\}_{i=1}^d$ of $\Omega$, let $\mathcal{V} := \bigcup_{i=1}^d \ext(P^i)$ and $\mathcal{S} := \{S^i\}_{i=1}^d$ where $S^i = \mathcal{V} \cap P^i$. The formulation~\eqref{form:cdc_tree} for $\CDC(\mathcal{S})$ uses $2(d - 1)$ fewer continuous variables than \eqref{form:cdc_vielma} or \eqref{form:cdc_tree_2}.
\end{theorem}

\begin{proof}
Let $\mathcal{S}'$, $\alpha$, and $\mathcal{T}'$ be the output of Algorithm~\ref{alg:build_CDC} given $\CDC(\mathcal{S})$ and $\disjoint := \False$ as inputs. As mentioned in Remark~\ref{rm:tree2}, \eqref{form:cdc_vielma} or \eqref{form:cdc_tree_2} introduces $w(\mathcal{T}')$ more continuous variables than \eqref{form:cdc_tree}. Let $\mathcal{K}_{\mathcal{S}}$ be the intersection graph of the disjunctive constraint $\CDC(\mathcal{S})$ and $\mathcal{T}$ is a maximum spanning tree of $\mathcal{K}_{\mathcal{S}}$. Since $\mathcal{T}'$ is a tree by replacing all the vertices $S^i$ to $S'^{,i}$ in the tree $\mathcal{T}$, $w(\mathcal{T}') = w(\mathcal{T})$. Since $\{P^i\}_{i=1}^d$ is a polytope partition of $\Omega \subseteq \mathbb{R}^2$, $|S^i \cap S^j| \leq 2$ for any $i \neq j$. Hence, $w(\mathcal{T}) \leq 2 (d-1)$.

On the other hand, let $\mathcal{D}$ be the dual graph of $\{P^i\}_{i=1}^d$. Since $\{P^i\}_{i=1}^d$ is a connected polytope partition of $\Omega$, $\mathcal{D}$ is connected. Thus, there exists a spanning tree subgraph of $\mathcal{D}$ and its corresponding spanning tree of $\mathcal{K}_{\mathcal{S}}$ has a weight of $2(d-1)$. Since $\mathcal{T}$ is a maximum spanning tree of $\mathcal{K}_{\mathcal{S}}$, $w(\mathcal{T}) \geq 2 (d-1)$.
\qed\end{proof}

Note that $S^i$ might contain elements other than all extreme points of $P^i$. In computational geometry, partitioning the obstacle-free region $\Omega \in \mathbb{R}^2$ to a smallest number of triangles (polytopes with three extreme points) is one of the most well-studied polygon partition problems~\cite{de2000computational}. In this setting, our new ideal formulation~\eqref{form:cdc_tree} uses $d+2$ continuous variables whereas \eqref{form:cdc_vielma} or \eqref{form:cdc_tree_2} uses $3d$ continuous variables.

\begin{corollary} \label{cor:triangle_part}
Given the same setting as Theorem~\ref{thm:poly_part}, let $P^i$ be a polytope with three extreme points for all $i \in \llbracket d \rrbracket$. The formulation~\eqref{form:cdc_tree} uses $d+2$ continuous variables whereas \eqref{form:cdc_vielma} or \eqref{form:cdc_tree_2} uses $3d$ continuous variables.
\end{corollary}








\section*{Acknowledgements}

The authors would like to thank Hamidreza Validi for helpful discussion and comments.

\bibliography{mybibfile}

\newpage
\appendix

\section{Proof for Lemma~\ref{lm:bc_size}} \label{ap:lm_proof}

\begin{proof}
Since $k$ is a finite positive integer, then there exists an integer $b'$ such that $k \leq 2^{b'} + 1$. Thus, the total number of the bicliques is no larger than
\begin{subequations}
\begin{align}
    \sum_{i=0}^{b-1} \min\left\{2^i, \left\lceil \frac{k-1+2^{b-i-1}}{2^{b-i}} \right\rceil\right\} & \leq \sum_{i=0}^{b-1} \left\lceil \frac{k-1+2^{b-i-1}}{2^{b-i}} \right\rceil\\
    &= \sum_{i=1}^{b} \left\lceil \frac{k-1+2^{i-1}}{2^{i}} \right\rceil \label{eq:sosk_proof_1_1}\\
    &= b + \sum_{i=1}^{b} \left\lceil \frac{k-1-2^{i-1}}{2^{i}} \right\rceil \\
    &\leq b + \sum_{i=1}^{b'} \left\lceil \frac{k-1-2^{i-1}}{2^{i}} \right\rceil \label{eq:sosk_proof_1_2}
\end{align}
\end{subequations}

We switch the value of $i$ to $b-i$ in \eqref{eq:sosk_proof_1_1}. In order to obtain~\eqref{eq:sosk_proof_1_2}, we prove it in two cases. If $b' \geq b$, since $\left\lceil \frac{k-1-2^{i-1}}{2^{i}} \right\rceil \geq \left\lceil -\frac{1}{2} \right\rceil$ for all positive integer $i$, it is trivial to show the correctness. If $b' < b$,
\begin{subequations}
\begin{align}
    \sum_{i=1}^{b} \left\lceil \frac{k-1-2^{i-1}}{2^{i}} \right\rceil &= \sum_{i=1}^{b'} \left\lceil \frac{k-1-2^{i-1}}{2^{i}} \right\rceil + \sum_{i=b'+1}^{b} \left\lceil \frac{k-1-2^{i-1}}{2^{i}} \right\rceil \\
    &\leq \sum_{i=1}^{b'} \left\lceil \frac{k-1-2^{i-1}}{2^{i}} \right\rceil + \sum_{i=b'+1}^{b} \left\lceil \frac{k-1}{2^{b'+1}} - \frac{1}{2} \right\rceil \\
    &= \sum_{i=1}^{b'} \left\lceil \frac{k-1-2^{i-1}}{2^{i}} \right\rceil \label{eq:sosk_proof_2_1},
\end{align}
\end{subequations}

\noindent where \eqref{eq:sosk_proof_2_1} is obtained by the fact that $k\leq 2^{b'} +1$. 

Since $k \leq 2^{b'} + 1$, then we can find $a_j \in \{0, 1\}$ for $j \in \{0, \hdots, b'-1\}$ such that $k-2 = \sum_{j=0}^{b'-1} a_j 2^j$. Then,
\begin{subequations}
\begin{align}
    \sum_{i=1}^{b'} \left\lceil \frac{k-1-2^{i-1}}{2^{i}} \right\rceil &= \sum_{i=1}^{b'} \left\lceil \frac{\left(\sum_{j=0}^{b'-1} a_j 2^j\right) + 1 -2^{i-1}}{2^{i}} \right\rceil \\
    &= \sum_{i=1}^{b'} \left\lceil \left(\sum_{j=0}^{b'-1} a_j 2^{j-i}\right) + 2^{-i} -\frac{1}{2} \right\rceil \\
    &= \sum_{i=1}^{b'-1} \sum_{j=i}^{b'-1} a_j 2^{j-i}+   \sum_{i=1}^{b'} \left\lceil \left(\sum_{j=0}^{i-1} a_j 2^{j-i}\right) + 2^{-i} -\frac{1}{2} \right\rceil \label{eq:sosk_proof_3_c}\\
    &= \sum_{i=1}^{b'-1} \sum_{j=i}^{b'-1} a_j 2^{j-i} + \sum_{i=1}^{b'}a_{i-1}. \label{eq:sosk_proof_3_d}
\end{align}
\end{subequations}
We take out all the integer parts to get~\eqref{eq:sosk_proof_3_c}. For~\eqref{eq:sosk_proof_3_d}, we can show the correctness by proving that $\left\lceil \left(\sum_{j=0}^{i-1} a_j 2^{j-i}\right) + 2^{-i} -\frac{1}{2} \right\rceil = a_{i-1}$. We prove it by cases. If $a_{i-1} = 0$,
\begin{align*}
    -\frac{1}{2} \leq \left(\sum_{j=0}^{i-1} a_j 2^{j-i}\right) + 2^{-i} -\frac{1}{2} \leq \sum_{j=0}^{i-2} 2^{j-i} + 2^{-i} - \frac{1}{2} = 0.
\end{align*}
If $a_{i-1} = 1$, then
\begin{align*}
    1 \geq \left(\sum_{j=0}^{i-1} a_j 2^{j-i}\right) + 2^{-i} -\frac{1}{2} \geq \frac{1}{2} + 2^{-i} - \frac{1}{2} = 2^{-i} > 0.
\end{align*}
Then, by reordering the summation,
\begin{subequations}
\begin{align}
    \sum_{i=1}^{b'-1} \sum_{j=i}^{b'-1} a_j 2^{j-i} + \sum_{i=1}^{b'}a_{i-1} &= \sum_{j=1}^{b'-1} a_j \sum_{i=1}^{j} 2^{j-i} + \sum_{j=0}^{b'-1}a_{j}\\
    &= a_0 + \sum_{j=1}^{b'-1}\left( a_j + a_j \sum_{i=0}^{j-1} 2^i\right) \label{eq:sosk_proof_4_a} \\
    &= a_0 + \sum_{j=1}^{b'-1}a_j 2^j\\
    &= k - 2.
\end{align}
\end{subequations}
\qed\end{proof}

\section{Proof for Propositions~\ref{prop:sosk_less} and~\ref{prop:sosk_proportion}} \label{ap:proof_sosk_number}

\begin{proof}
(Proof of Proposition~\ref{prop:sosk_less}) Given $N > k \geq 2$, we have
\begin{subequations}
\begin{align}
    \lceil \log_2(\lceil N / k \rceil - 1) \rceil + 3k &= \lceil \log_2(\lceil (N - k) / k \rceil) \rceil + 3k \label{eq:sosk_less_1}\\
    &\geq \lceil \log_2((N - k) / k) \rceil + 3k \label{eq:sosk_less_2} \\
    &= \lceil \log_2(N - k)  - \log_2 (k) \rceil + 3k \label{eq:sosk_less_3} \\
    &\geq \lceil \log_2(N - k)\rceil  - \lceil \log_2 (k) \rceil + 3k \label{eq:sosk_less_4} \\
    &\geq \lceil \log_2(N - k + 1)\rceil - 1 - \lceil \log_2 (k) \rceil + 3k \label{eq:sosk_less_5} \\
    &\geq \lceil \log_2(N - k + 1)\rceil - 1 + 2k \label{eq:sosk_less_6} \\
    &> \lceil \log_2(N - k + 1) \rceil + k - 2 \label{eq:sosk_less_7},
\end{align}
\end{subequations}
\noindent where \eqref{eq:sosk_less_4} is obtained by the fact that $\lceil x  - y \rceil \geq \lceil x  - \lceil y \rceil \rceil = \lceil x \rceil  - \lceil y \rceil$ for any $x, y$ and \eqref{eq:sosk_less_6} is because $\lceil \log_2(x)\rceil \geq \lceil \log_2(x + 1)\rceil - 1$ for any $x \geq 1$.

(Proof of Proposition~\ref{prop:sosk_proportion}) Furthermore, we suppose that $k > C \lceil \log_2(N) \rceil$ for some $C > \frac{1}{2}$. Then,

\begin{subequations}
\begin{align}
    \frac{\lceil \log_2(N - k + 1) \rceil + k - 2}{\lceil \log_2(\lceil N / k \rceil - 1) \rceil + 3k} &\leq \frac{\lceil \log_2(N) \rceil + k}{3 k} \\
    &< \frac{\frac{k}{C} + k}{3 k} \\
    &= \frac{1 + C}{3 C} \\
    &< 1.
\end{align}
\end{subequations}
\qed \end{proof}

\end{document}